\documentclass[11pt,reqno]{amsart}
\usepackage{color}
\usepackage[active]{srcltx}
\usepackage{a4wide}
\usepackage{amssymb, amsmath, mathtools}
\usepackage{graphicx}
\usepackage{float}
\usepackage[active]{srcltx}
\usepackage[pagewise]{lineno}
\usepackage[ansinew]{inputenc}
\usepackage{hyperref}

\theoremstyle{plain}
\newtheorem{theorem}{Theorem}[section]
\newtheorem{maintheorem}{Theorem}
\newtheorem{lemma}[theorem]{Lemma}
\newtheorem{proposition}[theorem]{Proposition}

\theoremstyle{remark}

\newtheorem{definition}{Definition}

\newtheorem{remark}[theorem]{Remark}

\numberwithin{equation}{section}

\newcommand{\NN}{{\mathbb{N}}}
\newcommand{\ZZ}{{\mathbb{Z}}}
\newcommand{\RR}{{\mathbb{R}}}
\newcommand{\EU}{{\mathbb{S}}}

\newcommand{\In}{{\text{In}}}
\newcommand{\Out}{{\text{Out}}}

\newcommand{\loc}{{\text{loc}}}

\newcommand{\dpt}{\displaystyle}

\begin{document}

\title[Strange attractors in the unfolding of a heteroclinic attractor]{``Large'' strange attractors \\ in the unfolding of a heteroclinic attractor}
\author[Alexandre Rodrigues]{Alexandre A. P. Rodrigues \\ Centro de Matem\'atica da Univ. do Porto \\ Rua do Campo Alegre, 687,  4169-007 Porto,  Portugal }
\address{Alexandre Rodrigues \\ Centro de Matem\'atica da Univ. do Porto \\ Rua do Campo Alegre, 687 \\ 4169-007 Porto \\ Portugal}
\email{alexandre.rodrigues@fc.up.pt}

\date{\today}

\thanks{AR was partially supported by CMUP (UID/MAT/00144/2019), which is funded by FCT with national (MCTES) and European structural funds through the programs FEDER, under the partnership agreement PT2020. AR also acknowledges financial support from Program INVESTIGADOR FCT (IF/00107/2015).}

\subjclass[2010]{ 34C28; 34C37; 37D05; 37D45; 37G35 \\
\emph{Keywords:} Heteroclinic bifurcations,  Bykov network, Strange attractors, Rank-one dynamics, Superstable sinks. }

\begin{abstract}
In this paper we present a  mechanism for the emergence of strange attractors in a one-parameter family of differential equations acting on a 3-dimensional sphere. When the parameter is zero, its flow exhibits an attracting heteroclinic network (Bykov network) made by two 1-dimensional connections and one 2-dimensional separatrix between hyperbolic saddles-foci with different Morse indices.  After slightly increasing  the parameter, while keeping the 1-dimensional connections unaltered, we concentrate our study in the case where the 2-dimensional invariant manifolds of the equilibria do not intersect. We will show that, for a set of parameters close enough to zero with positive Lebesgue measure, the dynamics exhibits strange attractors winding around an annulus in the phase space, supporting Sinai-Ruelle-Bowen (SRB) measures. We prove the existence of a sequence of parameter values for which the family  exhibits a superstable sink.   We also characterise the  transition from a Bykov network to a strange attractor.  
\end{abstract}

\maketitle \setcounter{tocdepth}{1}

\section{Introduction}\label{intro}

Homoclinic and heteroclinic bifurcations constitute the core of our understanding of complicated intermittent  behaviour in dynamical systems. It has started with Poincaré on the late XIX century, with subsequent contributions by the schools of Andronov, Shilnikov, Smale and Palis. These results rely on a combination of geometrical and analytical techniques used to understand the qualitative behaviour of the dynamics. 

Heteroclinic cycles and networks are flow-invariant sets that can occur robustly in dynamical systems and are frequently associated with intermittent behaviour. The rigorous analysis of the dynamics associated to the structure of the nonwandering sets close to heteroclinic networks is still a challenge. We refer to \cite{HS} for an overview of homoclinic bifurcations and for details on the dynamics near different types of heteroclinic strutures.
In this article, we establish connections between  the theory of rank-one attractors and a classical dynamical scenario related to heteroclinic attractors.
We present a mechanism that produce \emph{ ``large'' strange attractors}  in the terminology of Broer, Sim\'o and Tatjer \cite{BST98}: observable chaos which is not confined to a small portion of the phase space.

 \subsection{Strange attractors}
A compact attractor is
said to be \emph{strange} if it contains a dense orbit with at least one positive Lyapunov exponent. 
A dynamical phenomenon in a one-parameter family of maps is said to be \emph{persistent} if it occurs for a set of parameters of  positive Lebesgue measure.
Persistence of chaotic dynamics is physically relevant because it means that a given phenomenon is numerically observable with positive probability.

Strange attractors are of fundamental importance in dynamical systems; they have
been observed and recognized in many scientific disciplines  \cite{BIS, Gaspard, GH, Homb2002,  LR2015, RT71, TS1986}.  Atmospheric physics provides one of the most striking   examples of strange attractors observed in natural sciences. We address the reader to \cite{SNN95}  where the authors established the emergence of strange attractors in a low-order atmospheric circulation model. 
  Among the theoretical examples that have been studied are the Lorenz and H\'enon attractors, both of which are
closely related to suitable one-dimensional reductions.

The rigorous proof of the strange character of an invariant set is a great challenge and the proof of the persistence (in measure) of such attractors is an  involved task. 

For families of autonomous differential equations in $\RR^3$,  the persistence of strange attractors can be proved near homo or heteroclinic cycles whose first return map to a cross section exhibits a homoclinic tangency to a dissipative point \cite{{Homb2002}, LR2015, MV93, OS}.  In this paper we give a further step towards this analysis.  We provide a criterion for the existence of abundant  strange attractors (in the terminology of \cite{MV93}) near a specific heteroclinic configuration, using the \emph{theory of rank-one attractors}  developed by Q. Wang and L.-S. Young \cite{WY2001, WY, WY2003, WY2008}. This technique is quite general and may be applied other heteroclinic bifurcations with a single direction of instability.

\subsection{Rank-one attractors theory: an overview}
We briefly summarise how the theory of rank-one attractors  fits in the existing literature. 
In 1976, H\'enon \cite{He76} proposed the following two-parameter family of maps on $\RR^2$
\begin{equation}
\label{Henon}
f_{(a,b)}(x,y)=(1-ax^2+y, bx),
\end{equation}
for which numerical experiments for   $(a, b) = (1.4, 0.3)$ suggested the existence of a global attractor.  H\'enon conjectured that this dynamical system should have a strange attractor and that it might be more amenable for analysis than the Lorenz system. 

Benedicks and Carleson \cite{BC91} managed to prove that H\'enon's conjecture was true, not for the parameters $(a, b) = (1.4, 0.3)$  but for $b > 0$ small. In fact, for such small $b$-values, the map \eqref{Henon} is strongly dissipative, and may be seen as an ``unfolded'' version of the quadratic map on the interval.  It was shown that, for these values of $b$, there is a forward invariant region which  accumulates on a topological attractor that coincides with the topological closure of the unstable manifold of a  fixed point of saddle-type $p$, $\overline{W^u(p)}$. Results in \cite{BC91} state that, as long as $b > 0$ is kept sufficiently small, there is a positive Lebesgue measure set of parameters $a \in [1, 2]$ (close to $a = 2$) for which there is a dense orbit in  $\overline{W^u(p)}$ along which the derivative grows exponentially fast.  The techniques developed in \cite{BC91} promoted the emergence of several results not specific to the context of H\'enon maps, among which the work by Mora and Viana \cite{MV93} assumes a crucial importance. 

L. Mora and M. Viana \cite{MV93} proposed a \emph{renormalization scheme} that, when applied to a generic unfolding of a homoclinic tangency associated to a dissipative saddle, reveals the presence of H\'enon-like families. This means that chaotic attractors arise abundantly in  a specific dynamical scenario.
Continuing the study of the dynamical properties of H\'enon maps, Benedicks and Young \cite{BY93} developed these techniques  to obtain that every attractor occurring for suitable parameters $(a, b)$ of \eqref{Henon} supports a unique SRB measure.  Young \cite{Yo98} extended these results to dynamical systems that admit a horseshoe with infinitely many branches. In doing so, the author provided a \emph{general scheme} that unifies the proofs of these kinds of results in several dynamical situations.

The theory of rank-one maps, systematically developed by Wang and Young \cite{WY2001, WY, WY2006, WY2008}, concerns the dynamics of maps with some instability in one direction of the phase space and strong contraction in all other directions of the phase space. This theory   originated with the work of Jackobson \cite{Ja81} on the quadratic family and the  analysis of strongly dissipative H\'enon maps by Benedicks and Carleson \cite{BC91}. 

It is a comprehensive theory for a nonuniformly hyperbolic setting that is flexible enough to be applicable to concrete systems of differential equations and has experienced unprecedented
growth in the last 20 years in the context of non-autonomous systems. It provides checkable conditions that imply the existence of nonuniformly hyperbolic dynamics and  SRB measures in parametrized families $F_{\lambda}$ of dissipative embeddings in $\RR^n$ for $n \geq 2$. Roughly speaking, the theory asserts that, under certain checkable conditions, there exists a set $\Delta\subset \RR$ of values with positive Lebesgue measure such that if $\lambda \in \Delta$, then $F_\lambda$ has a strange attractor supporting a SRB measure.  
This theory has already been applied to several non-autonomous dynamical scenarios, including systems with stable foci and  limit cycles subject to pulsate drives  \cite{LWY, Ott2008, OS2010, WY, WY2003} and heteroclinic bifurcations \cite{MO15, Rodrigues2021,  WO, Wang_2016}.  Although our setting is formulated to give rigorous results, the theory can also provide justification for various mathematical statements about the strange attractors found in  \cite{CastroR2020}.

\subsection{Structure of the article}
 Motivated by the bifurcation scenario involving a Bykov attractor \cite{LR2015, LR2016}, in Section \ref{s:setting}, we enumerate the main assumptions concerning the configuration of an attracting network, preceding the presentation of the main results in Section \ref{main results}. 
 In Section \ref{s: theory}, we give a descriptive summary of the rank-one attractors theory in dimension 2, after the introduction of a \emph{Misiurewicz-type} map. 
 We also introduce some important dynamical and ergodic concepts.   
 
 The coordinates and other notation  used in the rest of the article are presented in  Section~\ref{localdyn}. 
In Sections \ref{proof Th A}, \ref{proof Th B} and \ref{Proof: ThC}, we prove the main results of this paper.  In Section  \ref{proof Th B}, we add an extra subsection with some complementary   remarks. 
Finally, in Section \ref{s:mechanism},  we relate our results with others in the literature and we point out some bifurcations in the family of vector fields, emphazising  the role of the \emph{twisting number} in the sequel.   
 
 Throughout this paper, we have endeavoured to make a self contained exposition bringing together all topics related to the proofs. We  revive some useful results from the literature;  we hope this saves the reader the trouble of going through the entire length of some referred works to achieve a complete description of the theory. 
 We have drawn illustrative figures to make the paper easily readable.

\section{Preliminaries}
\label{Definitions}
To make this paper self-contained and readable, we recall some definitions and results on heteroclinic bifurcations, adapted to our purposes. For $\varepsilon>0$ small enough and $k\geq 4$, consider the one-parameter family of $C^k$--smooth autonomous differential equations
\begin{equation}
\label{general2}
\dot{x}=f_{ \lambda}(x)\qquad x\in \EU^3  \qquad \lambda \in [0, \varepsilon] 
\end{equation}
where $\EU^3$ denotes the unit 3-dimensional sphere, endowed with the $C^k$--topology. 


\subsection{Bykov network}
\label{het_phenomena}
Suppose that $O_1$ and $O_2$ are two hyperbolic saddle-foci of \eqref{general2} with different Morse indices (dimension of the unstable manifold). We say that there is a {\em heteroclinic cycle} associated to $O_1$ and $O_2$ if 
$$
W^u(O_1)\cap W^s(O_2)\neq \emptyset\qquad \text{and}  \qquad W^u(O_2)\cap W^s(O_1)\neq \emptyset.$$
 For $i, j \in \{1,2\}$, the non-empty intersection of $W^u(O_i)$ with $W^s(O_j)$ is called a \emph{heteroclinic connection} between $O_i$ and $O_j$, and will be denoted by $[O_i \rightarrow O_j]$. Although heteroclinic cycles involving equilibria are not a generic feature within differential equations, they may be structurally stable within families of vector fields which are equivariant under the action of a compact Lie group $\mathcal{G}\subset \mathbb{O}(n)$, due to the existence of flow-invariant subspaces \cite{GH}. 
 A heteroclinic cycle between two hyperbolic saddle-foci of different Morse indices, where the invariant manifolds coincide, is called a \emph{Bykov network} \cite{Bykov00}.

\subsection{Rotational horseshoe}
\label{ss: rotational horseshoe}
Let $\mathcal{H} $ stand for the infinite annulus $\mathcal{H} = \EU^1 \times [0,1]$, endowed with the usual inner product from $\RR^2$. We denote by $Homeo^+(\mathcal{H} )$ the set of homeomorphisms of the annulus which preserve orientation.
Given a homeomorphism $f :X \rightarrow X$  and a partition of $m\in \NN\backslash\{1\}$ elements $R_0,..., R_{m-1}$ of $X\subset \mathcal{H}$, the itinerary  function $$ {\Upsilon}: X \rightarrow \{0, ..., m-1\}^\ZZ= \Sigma_m$$ is defined by $$\Upsilon(x)(j)=k\quad   \Leftrightarrow \quad f^j(x)\in R_k, \quad \text{for every} \quad j\in \ZZ.$$  
Following \cite{PPS}, we say that a compact invariant set $\Lambda \subset \mathcal{H} $ of $f \in Homeo^+(\mathcal{H} )$ is a \emph{rotational horseshoe} if it admits a finite partition $P =\{R_0, ..., R_{m-1} \}$ with $R_i$ open sets of $\Lambda$ so that:
\begin{itemize}
\item the itinerary $\Upsilon$ defines a semi-conjugacy between $f|_\Lambda$ and the full-shift $\sigma: \Sigma_m \rightarrow \Sigma_m$, that is $\Upsilon  \circ f = \sigma \circ \Upsilon$ with $\Upsilon$ continuous and onto;
\medbreak
\item for any lift $F: \RR^2 \rightarrow \RR^2$ of $f$, there exist  $k>0$ and $m$ vectors $v_0, ...,v_{m-1} \in \ZZ \times \{0\}$ so that\footnote{The usual norm of $\RR^2$ is denoted by $\|\star\|$.}:
$$
\left\| (F^n(\hat{x})-\hat{x})  - \sum_{i=0}^n v_{\Upsilon(x)(i)}\right\| <k \qquad \text{for every} \qquad  \hat{x}\in \pi^{-1}(\Lambda), \quad n\in \NN,
$$
where $\pi:\RR^2\rightarrow \mathcal{H}$ denotes the usual projection map and $\hat{x} \in \pi^{-1}(\Lambda)$ is the lift of $x$; more details in the proof of Lemma 3.1 of \cite{PPS}. 
\end{itemize}
The existence of a rotational horseshoe for a map implies positive \emph{topological entropy}  at least  $ \log m $.

\subsection{Strange attractors and SRB measures}
Following \cite{LWY},
we formalize the notion of strange attractor  for a two-parameter family $F_{(a,b)}$ defined on  $M= \EU^1 \times [0,1]$, endowed with the induced topology. Recall that, if $A\subset M$ then $\overline{A}$ is the  topological closure of $A$. 

\medbreak
Let $F_{(a,b)}$ be an embedding  such that $F_{(a,b)} (\overline{U} )\subset U$ for some non-empty open set $U\subset M$. In the present work we refer to 
$$
{\Omega} =\bigcap_{m=0}^{+\infty}  F_{({a},b)}^m(\overline{U}).
$$
as an \emph{attractor} and $U$ as its \emph{basin}. The attractor $\Omega$ is \emph{irreducible} if it cannot be written as the union of two (or more) disjoint attractors.
\medbreak
We say that $\Omega$ is a \emph{strange attractor} for $F_{({a},b)}$ if, for Lebesgue-a.e $(x,y) \in U\subset M$, the orbit of $(x,y)$ has a positive
Lyapunov exponent. In other words, denoting by $\| \star \|$ the usual norm of $\RR^2$, we have: 
$$
\lim_{n \in \NN} \quad \frac{1}{n}  \log\|D  F_{({a},b)} ^n(x,y) \|>0.
$$
 
We say that $ F_{({a},b)}$ possesses a \emph{strange attractor supporting an ergodic SRB measure} $\nu$ if: \\
\begin{enumerate}
\item $ F_{({a},b)}$ has a strange attractor,\\
\item the conditional measures of $\nu$ on unstable manifolds are equivalent to the Riemannian volume on these leaves and \\
\item  for Lebesgue-a.e. $(x,y)\in U\subset M$ and  for every continuous function $\varphi : U\rightarrow \RR$, we have:
\begin{equation}
\label{limit2}
\lim_{n\in \NN} \quad \frac{1}{n} \sum_{i=0}^{n-1} \varphi \circ  F_{({a},b)}^i(x,y) = \int  \varphi \, d\nu.
\end{equation}

\end{enumerate}
\bigbreak
We say that $F_{({a},b)}$ converges  in distribution (with respect to $\nu$) to the normal distribution if, for every ergodic SRB measure $\nu$ and every H\"older continuous function $\varphi: \Omega \rightarrow \RR$, the sequence $\left\{\varphi\left( F_{({a},b)}^i\right) : i\in \NN \right\}$ obeys a \emph{central limit theorem}; in other words, if $\int \varphi \, d \nu = 0$, then the sequence
$
\frac{1}{\sqrt{m}} \sum_{i=0}^{m-1}\varphi \circ  F_{({a},b)}^i
$
converges (in distribution with respect to $\nu$) to the normal distribution. The variance of the limiting normal distribution is strictly positive unless $\varphi \circ  F_{({a},b)} = \Psi \circ  F_{({a},b)} -\Psi$  for some $\Psi $.

\medbreak
Suppose that  $ F_{({a},b)}$ possesses a strange attractor and support a unique ergodic SRB measure $\nu$. The dynamical system $(F, \nu)$ is \emph{mixing} if it is \emph{isomorphic to a Bernoulli shift}.

\section{Setting}
\label{s:setting}
We will enumerate the main assumptions concerning the configuration of an attracting heteroclinic network.  

\subsection{The organising center}
\label{ss:oc}
For $\varepsilon>0$ small enough and $r\geq 4$, consider the one-parameter family of $C^r$--smooth differential equations
\begin{equation}
\label{general2.1}
\dot{x}=f_{\lambda}(x)\qquad x\in \EU^3 \qquad  \lambda \in [0, \varepsilon] 
\end{equation}
where $\EU^3$ denotes the unit 3-dimensional sphere, endowed with the usual $C^r$--topology. Denote by $\varphi_{\lambda}(t,x)$, $t \in \RR$, the associated flow\footnote{Since $\EU^3$ is a compact set without boundary, the solutions of \eqref{general2.1} may be extended to $\RR$.}, satisfying the following hypotheses for $\lambda=0$:

\bigbreak
\begin{enumerate}
 \item[\textbf{(P1)}] \label{B1}  There are two different equilibria, say $O_1$ and $O_2$.
 \bigbreak
 \item[\textbf{(P2)}] \label{B2} The eigenvalues  of $Df_0(X)$ are:
 \medbreak
 \begin{enumerate}
 \item[\textbf{(P2a)}] $E_1$ and $ -C_1\pm \omega_1 i $ where $C_1>E_1>0, \quad \omega_1>0$, \quad for $X=O_1$;
 \medbreak
 \item[\textbf{(P2b)}] $-C_2$ and $ E_2\pm \omega_2 i $ where $C_2> E_2>0, \quad \omega_2>0$,  \quad for $X=O_2$.
 \end{enumerate}
 \end{enumerate}

 \bigbreak
The equilibrium point $O_1$ possesses a $2$-dimen\-sional stable and $1$-dimen\-sional unstable manifold that will be denoted by $W^s(O_1)$ and $W^u(O_1)$, respectively.  Dually,   $O_2$ possesses a 1-dimen\-sional stable and $2$-dimen\-sional unstable manifold and the terminology is  $W^s(O_2)$ and $W^u(O_2)$.   For $M\subset \EU^3$, denoting by $\overline{M}$ the topological closure of $M$, we also assume that:
 
 \begin{enumerate}
  \bigbreak
  \item[\textbf{(P3)}]\label{B3} The manifolds $\overline{W^u(O_2)}$ and $\overline{ W^s(O_1)}$ coincide and   $\overline{W^u(O_2)\cap W^s(O_1)}$ consists of a two-sphere (also called the $2D$-connection). \bigbreak
\end{enumerate}

   \begin{enumerate}
\item[\textbf{(P4)}]\label{B4} There are two trajectories, say  $\gamma_1, \gamma_2$, contained in  $W^u(O_1)\cap W^s(O_2)$, one in each connected component of $\EU^3\backslash \overline{W^u(O_2)}$ (also called the $1D$-connections).
\end{enumerate}
 \bigbreak


For $\lambda=0$, the two equilibria $O_1$ and $O_2$, the 2-dimensional heteroclinic connection from $O_2$ to $O_1$ referred in \textbf{(P3)} and the two trajectories listed in  \textbf{(P4)}  build a \emph{heteroclinic network} we will denote hereafter by $\Gamma$. This network consists of two cycles and has an \emph{attracting} character \cite{LR}, this is why it will be called a \emph{Bykov\footnote{The terminology Bykov is a tribute to V. Bykov who  dedicated his latest research works to cycles similar to these ones.} attractor}. 

We say that $\Gamma\subset \EU^3$ is  \emph{asymptotically stable} if we may find an open neighbourhood $\mathcal{U}$ of the heteroclinic network $\Gamma$ having its boundary transverse to the flow of $\dot{x}=f_{0}(x)$ and such that every solution starting in $\mathcal{U}$ remains in it and is forward asymptotic to $\Gamma$. 
  
  \begin{lemma}[\cite{LR}]
\label{attractor_lemma}
The set $\Gamma$ is asymptotically stable.
\end{lemma}

\medbreak
There are two different possibilities for the geometry of the flow around $\Gamma$,
depending on the direction trajectories turn around the 1-dimensional heteroclinic trajectories from $O_1$ to $O_2$. To make this rigorous, we need some new concepts adapted from \cite{LR2015}.
Let $V_1$
and $V_2$ be small disjoint neighbourhoods of $O_1$ and $O_2$ with disjoint boundaries $\partial V_1$ and $\partial V_2$, respectively. These neighbourhoods have been constructed with detail in Section~\ref{localdyn}.
Trajectories starting at $\partial V_1\backslash W^s(O_1)$ near $W^s(O_1)$ go into the interior of $V_1$ in positive time, then follow one of the solutions in $[O_1 \rightarrow O_2]$, go inside $V_2$,  come out at $\partial V_2$ and then return to $\partial V_1$. This trajectory is not closed since $\Gamma$ is attracting. 
Let $\mathcal{Q}$ be a piece of trajectory like this from $\partial V_1$ to $\partial V_1$. Within $\partial V_1\backslash W^s(O_1)$, join its starting point to its end point by a segment, forming a closed curve, which we call the  \emph{loop} of $\mathcal{Q}$. 
By construction, the loop of $\mathcal{Q}$ and  $\Gamma$ are disjoint closed sets. 

\begin{definition}(\cite{LR2015})
We say that the two saddle-foci $O_1$ and $O_2$ in $\Gamma$ have the same \emph{chirality} if the loop of every trajectory is linked to $\Gamma$ in the sense that the two closed sets cannot be disconnected by an isotopy. Otherwise, we say that $O_1$ and $O_2$ have different chirality.
\end{definition}
 \medbreak
 From now on, we assume the following technical condition:
\medbreak
\begin{enumerate}
\item[\textbf{(P5)}] \label{B5} The saddle-foci $O_1$ and $O_2$ have the same chirality.
\end{enumerate}
\medbreak

\subsection{Perturbing term}
When the system \eqref{general2.1} is slighlty perturbed, the  2-dimensional invariant manifolds are generically transverse (either intersecting or not), as a consequence of the Kupka-Smale Theorem. With respect to the effect of the parameter $\lambda$ on the dynamics, we assume that:
 \medbreak
\begin{enumerate}
\item[\textbf{(P6)}] \label{B5} For $ \lambda > 0$,  the two trajectories within $W^u(O_1)\cap W^s(O_2)$ persist.
\end{enumerate}

 \medbreak
\begin{enumerate}
\item[\textbf{(P7)}]\label{B7} For $\lambda > 0$, the 2-dimensional manifolds $W^u(O_2)$ and $W^s(O_1)$ do not intersect.
\end{enumerate}

 \medbreak

 and
\medbreak

\begin{enumerate}
\item[\textbf{(P8)}] \label{B8} There exist $\varepsilon, \lambda_1>0$ for which the global maps associated to  the connections $[O_1 \rightarrow  O_2]$ and $[O_2 \rightarrow  O_1]$ are given, in local coordinates, by the \emph{Identity map} and by  $$(x,y)\mapsto (x + \xi + \lambda \Phi_1(x, y),y+ \lambda \Phi_2 (x, y)) \qquad \text{for} \qquad \lambda \in [0, \lambda_1]$$ respectively, where $\xi\in \RR,$ $$\Phi_1:\EU^1 \times [-\varepsilon, \varepsilon] \rightarrow \RR, \qquad \Phi_2: \EU^1 \times [-\varepsilon, \varepsilon] \rightarrow \RR^+$$ are $C^3$--maps and $\ln(\Phi_2(x, 0))$ is a Morse function with finitely many nondegenerate critical points\footnote{In Section \ref{Proof: ThC}, we list the  critical points of $\ln(\phi_2(x,0))$ as $c^{(1)},\ldots, c^{(q)}$, for some $q\in \NN \backslash \{2\}$.}. This assumption will be detailed later in Section \ref{localdyn}.
\end{enumerate}
\medbreak

 For  $r \geq 4$, denote by  $\mathfrak{X}_{Byk}^r(\EU^3)$, the family of $C^r$--vector fields on $\EU^3$ endowed with the $C^r$--Whitney topology, satisfying Properties \textbf{(P1)--(P8)}.

\subsection{Digestive remarks about the Hypotheses}
In this subsection, we discuss the Hypotheses  \textbf{(P1)--(P8)}, pointing out that they appear in specific contexts.
\bigbreak
\begin{remark}
The full description of the bifurcations associated to the Bykov attractor is a phenomenon of codimension three \cite{Algaba, KLW, LR}.  Nevertheless,  the setting described by \textbf{(P1)--(P8)}  is natural in $\textbf{SO}(2)$--symmetry-breaking contexts \cite{Aguiar_tese, LR2016} and also in the setting of some unfoldings of the Hopf-zero singularity \cite{BIS, DIKS, Gaspard}.  Hypothesis \textbf{(P6)} corresponds to the \emph{partial symmetry-breaking} considered in Section 2.4 of \cite{LR}.
The setting described by \textbf{(P1)--(P8)} generalizes Case (4) of \cite{Rodrigues2019}.
\end{remark}
\bigbreak
\begin{remark}
An explicit example of a vector field on $\EU^3$ satisfying Properties \textbf{(P1)--(P7)} may be found in Section 4.1.3.2 of  Aguiar \cite{Aguiar_tese}. Adding a generic term to system (4.24) of  \cite{Aguiar_tese}, a 2-dimensional torus will break generically and our results of Section \ref{main results} are valid. A numerical analysis of this system has been performed in \cite{CastroR2020}.
\end{remark}
\bigbreak

\begin{remark}
In the simplest scenario for the splitting of the   sphere defined by the coincidence of the two-dimensional invariant manifolds, one observes either  heteroclinic tangles or empty intersection. The expression for the global map  associated to $[O_2 \rightarrow  O_1]$   is generic in the second case. 
 The distance between $W^u (O_2)$ and $W^s (O_1)$ in a  given cross section to $\Gamma$ may be computed using the  \emph{Melnikov integral} \cite{GH}.

  \end{remark}

\bigbreak

\begin{remark}
The analytical  expressions for the transitions maps along the connections $[O_1 \rightarrow  O_2]$ and $[O_2 \rightarrow  O_1]$ could be written as a general  \emph{Linear map}  as the one considered in \cite{Bykov00}:
$$
\left(\begin{array}{c} x  \\ y\end{array}\right) \mapsto \left(\begin{array}{cc} a  &0\\ 0&\frac{1}{a}\end{array}\right)\left(\begin{array}{c} x  \\ y\end{array}\right) ,
$$
and by  
$$
\left(\begin{array}{c} x  \\ y\end{array}\right) \mapsto \left(\begin{array}{c} \xi_1  \\ \xi_2 \end{array}\right)+ \left(\begin{array}{cc} b_{1} & b_2  \\ c_1 & c_2 \end{array}\right) \left(\begin{array}{c} x  \\ y\end{array}\right) +\lambda \left(\begin{array}{c} \Phi_1(x,y)  \\ \Phi_2(x,y)\end{array}\right) 
$$
respectively, where $a\geq 1$, $\xi_1, \xi_2, b_1, b_2, c_1, c_2 \in \RR$. For the sake of simplicity and, without loss of generality,  we restrict to the case $a=1= b_1=c_2=1$ and $\xi_2= b_2=c_1=0$. This simplifies the computations and it is not a restriction (see Section 6 of \cite{Rodrigues3}). 
\end{remark}

\bigbreak

\subsection{Constants}
\label{constants1}
For future use, we settle the following notation on the \emph{saddle-value} of $O_1$, $O_2$ and $\Gamma$:
\begin{equation}
\label{constants}
\delta_1 = \frac{C_1}{E_1 }>1, \qquad \delta_2 = \frac{C_2}{E_2 }>1, \qquad \delta=\delta_1\, \delta_2>1 
 \end{equation}
and on the \emph{twisting number} defined as:
\begin{equation}
\label{constants2}
  K_\omega= \frac{E_2 \, \omega_1+C_1\, \omega_2 }{E_1E_2}>0.
 \end{equation}

\section{Presentation of the results}\label{main results}

Let $\mathcal{T}$ be a neighborhood of the Bykov attractor $\Gamma$ that exists for $\lambda=0$.  For $\lambda_0>0$ small enough and $r\geq4$, let $\left(f_{\lambda}\right)_{\lambda \in \, ]0, \lambda_0] }$ be a one-parameter family of vector fields in $\mathfrak{X}_{Byk}^r(\EU^3)$ satisfying conditions \textbf{(P1)--(P8)}. 
\bigbreak
\begin{maintheorem}\label{thm:0}
Let $f_{\lambda} \in\mathfrak{X}_{Byk}^r(\EU^3)$ with $r\geq 4$. 
Then, there is $\tilde\varepsilon>0$ (small) such that the first return map to a given cross section $\Sigma$ to $\Gamma$ may be written (in local coordinates of $\Sigma$)   by:
$$
\mathcal{F}_{\lambda}(x,y)=\left[ x+\xi+\lambda \Phi_1(x,y) - K_\omega  \ln (y+\lambda\Phi_2(x,y)) \pmod{2\pi}, \, \, \, (y+\lambda\Phi_2(x,y))^\delta \right]+\dots
$$
where $\xi\in \RR$ and $$(x,y)\in \mathcal{D}=\{x\in \RR \pmod{2\pi}, \quad y/\tilde \varepsilon  \in [-1, 1] \quad \text{and} \quad y+\lambda\Phi_2(x,y)> 0\}$$ and the ellipses stand for asymptotically small  terms depending on $x$ and $y$ which converge to zero along with their derivatives.

\end{maintheorem}
\bigbreak

The proof of Theorem \ref{thm:0} is done in Section \ref{proof Th A} by composing local and global maps. 
Since $\delta>1$, for $\lambda$ small enough, the second component of $\mathcal{F}_{\lambda}$ is contracting and the dynamics of $\mathcal{F}_{\lambda}$ is dominated by the \emph{family of circle maps}
$$
h_a(x)= \theta + a +\xi + K_\omega \ln |\phi_2(x,0)|$$
where $\xi\in \RR$,  $x\in \EU^1$ and $ a  \sim - K_\omega \ln \mu \pmod{2\pi}$. Next result shows that, for any small smooth unfolding of $f_0$, in the $C^4$--Whitney topology, there is a sufficiently large \emph{twisting number} $K_\omega$ which forces the emergence of a strange attractor (for $\mathcal{F}_{\lambda}$ defined in Theorem  \ref{thm:0}).

\bigbreak
\begin{maintheorem}\label{thm:B}
Let $f_{\lambda} \in\mathfrak{X}_{Byk}^r(\EU^3)$ with $r\geq 4$. 
 Then there exists $K_\omega^0 > 0$ such that if $K_\omega>K_\omega^0$,   there exists a set $\tilde\Delta\subset [0, \lambda_0]$ of positive Lebesgue measure with the following properties:
  for every $\lambda \in \tilde\Delta$ the map $\mathcal{F}_\lambda$ exhibits an irreducible strange attractor that supports a unique ergodic SRB measure $\nu$. 
  The orbit of Lebesgue almost all points on the cross section $\Sigma$ has a positive Lyapunov exponent and is asymptotically distributed according to $\nu$. 
\end{maintheorem}
\bigbreak
The  $\mathcal{F}_\lambda$--iterates of almost all points on the cross section $\Sigma$ winds around an annulus (as discussed in Section \ref{s:mechanism}), justifying the title of this manuscript on \emph{``large'' strange attractors} according to the terminology of \cite{BST98}.  The dynamics are genuinely non-uniformly hyperbolic, a central limit theorem holds and correlations decay at an exponencial rate \cite{WY}.
The proof of Theorem \ref{thm:B} is performed  in Section \ref{proof Th B} by reducing the dynamics of the 2-dimensional first return map to the dynamics of a one-dimensional map, via the theory of \emph{rank-one attractors} described in Section \ref{s: theory}.  
\bigbreak

Following the reasoning  of \cite{Homb2002}, the next novelty   is the existence of a sequence of parameters converging to zero for which the flow of \eqref{general2.1}  exhibits a \emph{superstable 2-periodic    orbit} (\emph{i.e.} the   map $\mathcal{F}_{\lambda}$ has a 2-periodic point which is \emph{critical}).

 \bigbreak
\begin{maintheorem}\label{thm:C}
Let $f_{\lambda} \in\mathfrak{X}_{Byk}^r(\EU^3)$ with $r\geq 4$. There exists $K_\omega^0 > 0$  and a sequence $(\lambda_n)_{n\in \NN}$ converging to $0$ such that, if  $K_\omega>K_\omega^0$ then the flow of $\dot{x}=f_{\lambda_n}(x)$   exhibits a superstable 2-periodic    orbit.
 \end{maintheorem}
Each one one these sinks persist within intervals around $\lambda=\lambda_n$, $n\in \NN$. 
 The proof of Theorem \ref{thm:C} is performed in Section \ref{Proof: ThC}.

\medbreak

\section{Theory of rank-one maps}
\label{s: theory}

We gather in this section a collection of technical facts used repeatedly in later sections. In what follows, let us denote by $C^2(\EU^1,\RR) $ the set of $C^2$--maps from $\EU^1$ (unit circle) to $\RR$. For $h\in  C^2(\EU^1,\RR) $, let 
$$  C(h)= \{x \in \EU^1 : h'(x) = 0\}$$ 
be the \emph{critical set} of $h$. For $\delta>0$, let $C_\delta$ be the $\delta$--neighbourhood of $C(h)$ in $\EU^1$ and let $C_\delta(c)$ be the $\delta$--neighbourhood of $c\in C(h)$. The terminology \emph{dist} denotes the euclidian  metric on $\RR$.

\subsection{Misiurewicz-type map}
\label{Misiurewicz-type map}
We say that $h\in  C^2(\EU^1,\RR) $ is a \emph{Misiurewicz-type map}   if the following assertions hold:
\bigbreak
\begin{enumerate}
\item There exists $\delta_0>0$ such that: \\
\begin{enumerate}
\item $\forall x \in C_{\delta_0}$, we have $h''(x)\neq 0$ and\\
\item $\forall c\in C(h)$ and $n\in \ZZ^+$, $dist(h^n(c), C(h))\geq \delta_0$.\\
\end{enumerate}

\bigbreak

\item There exist constants $b_0, \lambda_0 \in \RR^+$ such that for all $\delta<\delta_0$ and $n\in \NN$, we have: \\
\begin{enumerate}
\item if $h^k(x)\notin C_\delta$ for $k\in\{0, ,..., n-1\}$, then $|(h^n)'(x)|\geq b_0\,  \delta\,  \exp(\lambda_0\, n)$.\\
\item  if $h^k(x)\notin C_\delta$ for $k\in\{0, ,..., n-1\}$ and $h^n(x)\in C_{\delta_0} $, then $|(h^n)'(x)|\geq b_0\,     \exp(\lambda_0\, n)$.\\
\end{enumerate}
\bigbreak
\end{enumerate}

For $\delta>0$, the set  $\EU^1$  may be divided into two regions: $C_{\delta}$ and $\EU^1 \backslash C_{\delta}$. In $\EU^1 \backslash C_{\delta}$, $h$ is essentially uniformly expanding; in $C_{\delta}\backslash C$,  although $|h '(x)|$ is small, the orbit of $x$ does not return to $C_{\delta}$   until its derivative has regained an amount of exponential growth. This kind of maps is a slight generalisation of those studied by Misiurewicz \cite{M}.

 \begin{figure}[h]
\begin{center}
\includegraphics[height=7cm]{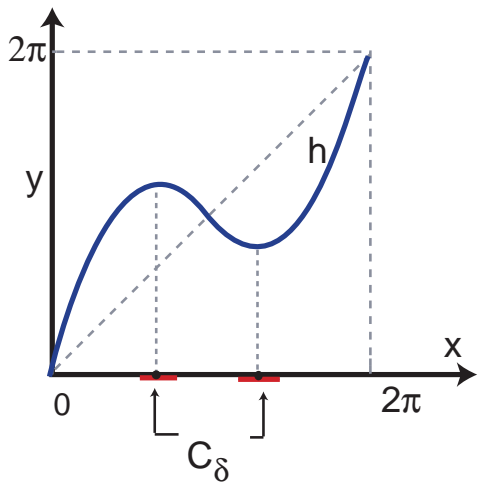}
\end{center}
\caption{\small  Example of a Misiurewicz-type map $h:\EU^1 \rightarrow \RR$. For $\delta>0$, the set  $C_{\delta}$  is a neighbourhood of the set of critical points $C$.  These maps are among the simplest examples with non-uniform expansion. }
\label{misiurewicz1}
\end{figure}

\subsection{Rank-one maps}
\label{rank_one}

Let $M= \EU^1 \times [0,1]$, induced with the usual topology. We consider the two-parameter family of maps $F_{(a,b)}: M \rightarrow M$, where $a \in[0, 2\pi]$ where $0$ and $2\pi$ are identified, and $b \in \RR$ is a scalar. Let $B_0 \subset \RR\backslash\{0\}$ with  0 as an accumulation point. Rank-one theory asks   the following hypotheses:
\bigbreak

\begin{description}
\item[\text{(H1) Regularity conditions}]  
\begin{enumerate}  
\item For each $b\in B_0$, the function $(x,y,a)\mapsto F_{(a,b)}$ is at least $C^3$--smooth.
\item Each map $F_{(a,b)}$ is an embedding of $M$ into itself.
\item There exists $k\in \RR^+$ independent of $a$ and $b$ such that for all $a \in  [0,2\pi]$, $b\in B_0$ and $(x_1, y_1), (x_2, y_2) \in M$, we have:
$$
\frac{|\det DF_{(a,b)}(x_1, y_1)|}{|\det DF_{(a,b)}(x_2, y_2)|} \leq k.
$$
\end{enumerate}

\bigbreak
\item[\text{(H2) Existence of a singular limit}] 
For $a\in [0,2\pi]$, there exists a map $$F_{(a,0)}: M \rightarrow \EU^1 \times \{0\}$$ such that the following property holds: for every $(x, y) \in M$ and $a\in [0,2\pi]$, we have
$$
\lim_{b \rightarrow 0} F_{(a,b)}(x,y) = F_{(a,0)}(x,y).
$$

\bigbreak
\item[\text{ (H3) $C^3$--convergence to the singular limit}]   For every choice of $a\in [0,2\pi]$, the maps $(x,y,a) \mapsto F_{(a,b)}$  converge in the $C^3$--topology to $(x,y,a) \mapsto F_{(a,0)}$ on $M\times [0,2\pi]$ as $b$ goes to zero.

\bigbreak
\item[\text{(H4) Existence of a sufficiently expanding map within the singular limit}]

The\-re exists $a^\star \in [0, 2\pi]$  such that $h_{a^\star}(x)\equiv F_{(a^\star, 0)}(x,0)$ is a Misiurewicz-type map (see Subsection \ref{Misiurewicz-type map}).

\bigbreak
\item[\text{(H5) Parameter transversality}] Let $  C_{a^\star}$ denote the critical set of a Misiurewicz-type map $h_{a^\star}$.
For each $x\in C_{a^\star}$, let $p = h_{a^\star}(x)$, and let $ {x(\tilde{a})}$ and $ {p(\tilde{a})}$ denote the 
continuations of $x$ and $p$, respectively, as the parameter $a$ varies around $a^\star$. The point $ {p(\tilde{a})}$  is the unique point such that $ {p(\tilde{a})}$  and $p$ have identical symbolic itineraries under $h_{a^\star}$ and $h_{\tilde{a}}$, respectively. We have:
$$
\frac{d}{da} h_{\tilde{a}}(x(\widetilde{a}))|_{a=a^\star} \neq \frac{d}{da} p(\tilde{a})|_{a=a^\star}.
$$

\bigbreak

\item[\text{(H6) Nondegeneracy at turns}] For each $x\in C_{a^\star}$, we have
$$
\frac{d}{dy} F_{(a^\star,0)}(x,y) |_{y=0} \neq 0.
$$

\bigbreak

\item[\text{(H7) Conditions for mixing}] If $J_1, \ldots, J_r$ are the intervals of monotonicity of a Misiu\-rewicz-type map $h_{a^\star}$, then:
\medbreak
\begin{enumerate}
\item $\exp(\lambda_0/3)>2$ (see the meaning of $\lambda_0$ in Subsection \ref{Misiurewicz-type map}) and
\medbreak
\item if $Q=(q_{im})$ is the matrix of all possible transitions  defined by:
\begin{equation*}
\left\{
\begin{array}{l}
1 \qquad \text{if} \qquad J_m\subset h_{a^\star} (J_i)\\  
0 \qquad \text{otherwise},\\
\end{array}
\right.
\end{equation*}
then there exists $N\in \NN$ such that $Q^N>0$ (\emph{i.e.} all entries of the matrix $Q^N$, endowed with the usual product,  are positive).
\bigbreak
\end{enumerate}

\end{description}

\begin{definition}
\label{terminologia1}
Identifying $\EU^1 \times \{0\}$ with $\EU^1$, we refer to $F_{(a,0)}$  the restriction $h_a : \EU^1 \rightarrow \EU^1$ defined by $h_a(x) = F_{(a,0)}(x,0)$ as the \emph{singular limit} of $F_{(a,b)}$ (see \textbf{(H4)}).
\end{definition}


\subsection{Q. Wang and L.-S. Young's reduction}
For attractors with strong dissipation and one direction of instability, Wang and Young conditions \textbf{(H1)--(H7)} are relatively simple and checkable; when satisfied, they guarantee the existence of  strange attractors with a package of statistical and geometric properties as follows:

\begin{theorem}[\cite{WY}, adapted]
\label{th_review}
Suppose the family $F_{(a,b)}$ satisfies \textbf{(H1)--(H7)}. Then, for all sufficiently small $b\in  B_0$, there exists a subset $\Delta \subset  [0, 2\pi]$ with positive Lebesgue measure such that for $a\in \Delta$, the map $F_{({a},b)}$ admits an irreducible strange attractor $\tilde{\Omega}\subset \Omega$  that supports a unique ergodic SRB measure $\nu$. The orbit of Lebesgue almost all points in $\tilde{\Omega} $ has positive Lyapunov exponent  and is asymptotically distributed according to $\nu$. 
\end{theorem}

In contrast to earlier results, the theory in \cite{WY2001, WY, WY2003} is
generic, in the sense that the conditions under which it holds rely only to certain
general characteristics of the maps and not to specific formulas or contexts. 

\medbreak

\section{Local and transition maps}\label{localdyn}

In this section we will analyze the dynamics near the network $\Gamma$ through local maps, after selecting appropriate coordinates in neighborhoods of the saddle-foci $O_1$ and $O_2$.

\subsection{Local coordinates}
\label{ss:lc}
In order to describe the dynamics around the Bykov cycles of $\Gamma$, we use the local coordinates near the equilibria $O_1$ and $O_2$ introduced in  \cite{LR2016}. See also  Ovsyannikov and Shilnikov \cite{OS}.

\medbreak

In these coordinates, we  consider cylindrical neighbourhoods  $V_1$ and $V_2$  in ${\RR}^3$ of $O_1 $ and $O_2$, respectively, of radius $\rho=\varepsilon>0$ and height $z=2\varepsilon$.
After a linear rescaling of the variables, we  may also assume that  $\varepsilon=1$.
Their boundaries consist of three components: the cylinder wall parametrised by $x\in \RR\pmod{2\pi}$ and $|y|\leq 1$ with the usual cover $$ (x,y)\mapsto (1 ,x,y)=(\rho ,\theta ,z)$$ and two discs, the top and bottom of the cylinder. We take polar coverings of these disks $$(r,\varphi )\mapsto (r,\varphi , \pm 1)=(\rho ,\theta ,z)$$
where $0\leq r\leq 1$ and $\varphi \in \RR\pmod{2\pi}$.
The local stable manifold of $O_1$, $W^s(O_1)$, corresponds to the circle parametrised by $ y=0$. In $V_1$ we use the following terminology:
\begin{itemize}
\item
$\In(O_1)$, the cylinder wall of $V_1$,  consisting of points that go inside $V_1$ in positive time;
\item
$\Out(O_1)$, the top and bottom of $V_1$,  consisting of points that go outside $V_1$ in positive time.
\end{itemize}
We denote by $\In^+(O_1)$ the upper part of the cylinder, parametrised by $(x,y)$, $y\in[0,1]$ and by $\In^-(O_1)$ its lower part.

\medbreak
The cross-sections obtained for the linearisation around $O_2$ are dual to these. The set $W^s(O_2)$ is the $z$-axis intersecting the top and bottom of the cylinder $V_2$ at the origin of its coordinates. The set 
$W^u(O_2)$ is parametrised by $z=0$, and we use:


\begin{itemize}
\item
$\In(O_2)$, the top and bottom of $V_2$,  consisting of points that go inside $V_2$ in positive time;
\item
$\Out(O_2)$,  the cylinder wall  of $V_2$,  consisting of points that go inside $V_2$ in negative time, with $\Out^+(O_2)$ denoting its upper part, parametrised by $(x,y)$, $y\in[0,1]$ and $\Out^-(O_2)$  its lower part.
\end{itemize}

We will denote by $W^u_{\loc}(O_2)$ the portion of $W^u(O_2)$  that goes from $O_2$ up to $\In(O_1)$ not intersecting the interior of $V$ and by $W^s_{\loc}(O_1)$  the portion of $W^s(O_1)$ outside $V_2$ that goes directly  from $\Out(O_2)$ into $O_1$. The flow is transverse to these cross-sections and the boundaries of $V_1$ and of $V_2$ may be written as the closure of  $\In(O_1) \cup \Out (O_1)$ and  $\In(O_2) \cup \Out (O_2)$, respectively. The orientation of the angular coordinate near $O_2$ is chosen to be compatible with that the direction induced by Hypothesis \textbf{(P4)} and  \textbf{(P5)}.

 \bigbreak

\subsection{Local maps near the saddle-foci}
Following \cite{Deng1}, the trajectory of  a point $(x,y)$ with $y>0$ in $\In^+(O_1)$, leaves $V_1$ at
 $\Out(O_1)$ at
\begin{equation}
\mathcal{L}_{1 }(x,y)=\left(y^{\delta_1} + S_1(x,y;  \lambda),-\frac{\omega_1 \, \ln y}{E_1}+x+S_2(x,y;  \lambda) \right)=(r,\varphi),
\label{local_v}
\end{equation}
where $S_1$ and $S_2$ are smooth functions which depend on $\lambda$ and satisfy:
\begin{equation}
\label{diff_res}
\left| \frac{\partial^{k+l+m}}{\partial x^k \partial y^l  \partial \lambda ^m } S_i(x, y; \lambda)
\right| \leq C y^{\delta_1 + \sigma - l},
\end{equation}
and $C$ and $\sigma$ are positive constants and $k, l, m$ are non-negative integers. Similarly, a point $(r,\varphi)$ in $\In(O_2) \backslash W^s_{\loc}(O_2)$, leaves $V_2$ at $\Out(O_2)$ at
\begin{equation}
\mathcal{L}_2(r,\varphi )=\left(-\frac{\omega_2\, \ln r}{E_2}+\varphi+ R_1(r,\varphi ;  \lambda),r^{\delta_2 }+R_2(r,\varphi;  \lambda )\right)=(x,y)
 \label{local_w}
\end{equation}
where $R_1$ and $R_2$ satisfy a  condition similar  to (\ref{diff_res}). The terms $S_1$, $S_2$,  $R_1$, $R_2$ correspond to asymptotically small terms that vanish when the radial components $y$ and $r$ go to zero. \bigbreak

\subsection{The transitions}\label{transitions}
The coordinates on $V_1$ and $V_2$ are chosen so that $[O_1\rightarrow O_2]$ connects points with $z>0$ (resp. $z<0$) in $V_1$ to points with $z>0$  (resp. $z<0$) in $V_2$. Points in $\Out(O_1) \setminus W^u_{\loc}(O_1)$ near $W^u(O_1)$ are mapped into $\In(O_2)$ along a flow-box around each of the connections $[O_1\rightarrow O_2]$. We will assume that the transition
$$\Psi_{1 \rightarrow  2}\colon \quad \Out(O_1) \quad \rightarrow  \quad \In(O_2)$$
does not depend on  $\lambda$ and is the Identity map, a choice compatible with hypothesis \textbf{(P4)} and  \textbf{(P5)}. Denote by $\eta$ the map
$$\eta =\mathcal{L}_{2} \circ \Psi_{1 \rightarrow  2} \circ \mathcal{L}_{1 }\colon \quad \In(O_1) \quad \rightarrow  \quad \Out(O_2).$$
From \eqref{local_v} and \eqref{local_w}, omitting high order terms in $y$ and $r$,  we infer that, in local coordinates, for $y>0$ we have
\begin{equation}\label{eqeta}
\eta(x,y)=\left(x-K_\omega \log y \,\,\,\pmod{2\pi}, \,y^{\delta} \right)
\end{equation}
with
\begin{equation}\label{delta e K}
\delta=\delta_1 \delta_2>1 \qquad \text{and} \qquad  K_\omega= \frac{C_1\, \omega_2+E_2\, \omega_1}{E_1 \, E_2} > 0.
\end{equation}

 \begin{figure}[h]
\begin{center}
\includegraphics[height=5.7cm]{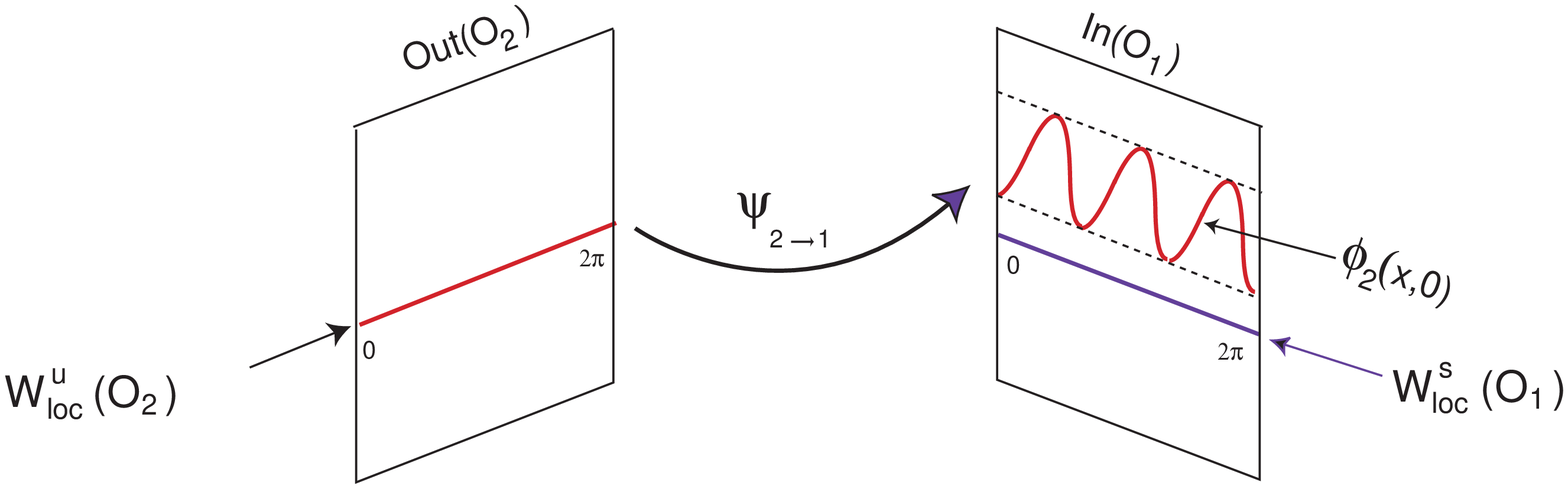}
\end{center}
\caption{\small  Illustration of the transition map  $\Psi_{2 \rightarrow  1} $ from $\Out(O_2)$ to $\In (O_1)$.  The graph of $\Phi_2(x,0)$ may be seen as the first hit of $W^u(O_2)$ to $\In (O_1)$.}
\label{transition_new}
\end{figure}

If $y<0$, solutions that are not trapped by the stable manifold of $O_1$, enter in $\In(O_1)$ after some positive time.  
Using  \textbf{(P7)} and \textbf{(P8)}, for $ \lambda \in \,  [0, \lambda_1]$, we still have a well defined transition map
$$\Psi_{2 \rightarrow  1}^{\lambda}:\Out(O_2)\rightarrow  \In(O_1)$$
that depends on the parameter $\lambda$, given by (see \textbf{(P8)}):
\begin{equation}\label{transition21}
\Psi_{2 \rightarrow  1}^{\lambda}(x,y)=\left(x + \xi +\lambda \Phi_1(x,y), \,y +\lambda \Phi_2(x,y) \right).
\end{equation}

 To simplify the notation, in what follows we will sometimes drop the subscript $ \lambda$, unless there is some  of misunderstanding.  
By \textbf{(P8)}, one knows that $\ln \Phi_2(x,0) $ has a finite number of nondegerate critical points -- see Figure \ref{transition_new}. Hereafter, let us collect them as $\{c^{(1)}, \ldots, c^{(q)}\}$ where $q\in \NN$, $c^{(i)}<c^{(i+1)}$ for $i\in \{1,...,q\}$ and $c^{(q+1)}\equiv c^{(1)}$. 

\section{Proof of Theorem \ref{thm:0}}
\label{proof Th A}
The proof of Theorem \ref{thm:0} is straightforward by composing the local and global maps constructed in Section \ref{localdyn}.  Let
\begin{equation}\label{first return 1}
\mathcal{F}_{\lambda} = \eta \circ  \Psi_{2 \rightarrow  1}^{\lambda}= \colon \quad  \mathcal{D} \subset \Out(O_2) \quad \rightarrow  \quad \mathcal{D} \subset \Out(O_2)
\end{equation}
be the first return map to $\mathcal{D}$, where $\mathcal{D}\subset \Sigma=\Out(O_2)$ is the set of initial conditions $(x,y) \in \Out(O_2)$ whose solution returns to $\Out(O_2)$. Up to high order terms, composing   $\eta$ (\ref{eqeta}) with $\Psi_{2 \rightarrow  1}^{\lambda}$ (\ref{transition21}), the analytic expression of $\mathcal{F}_{\lambda}$ is given by:
\label{first1}
\begin{eqnarray*}
\mathcal{F}_{\lambda}(x,y)&=& \left[  x+\xi+\lambda \Phi_1(x,y) - K_\omega  \ln (y+\lambda\Phi_2(x,y)) \pmod{2\pi}, \, \, \left(y +\lambda \Phi_2(x,y) \right)^\delta\right]\\
&=& \left(\mathcal{F}^1_{\lambda}(x,y), \mathcal{F}^2_{\lambda}(x,y)\right)
\end{eqnarray*}
The approximation of $\mathcal{F}_\lambda$ may be performed in a $C^3$--norm since the local maps $\mathcal{L}_1$ and $\mathcal{L}_2$   may be taken to be $C^{r-1}$ ($r\geq 4$ is the class of differentiability of the initial vector field) and the global maps are assumed to be $C^3$--embeddings.

 \begin{remark}
 When $\lambda = 0$, we may write (for $y> 0$):
 $$
 \mathcal{F}_{0}(x,y)= (  x+\xi - K_\omega  \ln y \pmod{2\pi}, \, \, y^\delta). $$
 This means that the $y$-component is contracting and thus the dynamics is governed by the $x$-component (circle maps). This is consistent to the fact that the network $\Gamma$ is asymptotically stable (Lemma \ref{attractor_lemma}). Notice that the remaining theory does not hold for $\lambda=0$ due to the change of coordinates \eqref{change1}.
 \end{remark}

\section{Proof of Theorem \ref{thm:B}}
\label{proof Th B}

\subsection{Insight into the reasoning}
In the proof  of Theorem \ref{thm:B}, we   transform the map $\mathcal{F}_\lambda$ into a two-parameter family of embeddings $\mathcal{F}_{(a,b)}$ satisfying the Hypotheses \textbf{(H1)--(H7)} of \cite{WY} (revisited in Subsection \ref{rank_one}). 
For a given  parameter $a\in [0, 2\pi[$, we construct a sequence $(b_n)_n$ of $\lambda$-values such that the singular limit is well defined.   If this  one-dimensional map has certain ``good'' properties, then some of them can be passed back to the 2-dimensional system ($b > 0$). They in
turn allow us to prove results on strange attractors for a positive Lebesgue measure set
of $\lambda$. Our reasoning is similar to the one used in Section 3 of \cite{WO}.

\subsection{Change of coordinates}
\label{change_of_coordinates}
For $\lambda \in\,\,  ]0, \lambda_1[ $ fixed and $(x,y) \in \Out(O_2)$, let us make the following change of coordinates:
  \begin{equation}
  \label{change1}
  \overline{x} \mapsto {x} \qquad \text{and} \qquad \overline{y} \mapsto \frac{y}{\lambda}.
  \end{equation}
Taking into account that: 
\begin{eqnarray*}
\mathcal{F}^1_{\lambda} (x,y) &=& x+\xi+{\lambda} \Phi_1(x,y) - K_\omega  \ln (y+{\lambda}\Phi_2(x,y)) \pmod{2\pi}\\
&=&   x+\xi+{\lambda} \Phi_1(x,y) - K_\omega  \ln \left[{\lambda} \left(\frac{y}{\lambda }+\Phi_2(x,y)\right)\right] \pmod{2\pi}\\
 &=&   x+\xi+{\lambda}  \Phi_1(x,y) - K_\omega  \ln \lambda -K_\omega \ln  \left[\left(\frac{y}{\lambda}+\Phi_2(x,y)\right)\right] \pmod{2\pi}
  \end{eqnarray*}
  and
    \begin{eqnarray*}
\mathcal{F}^2_{\lambda} (x,y) &=& (y+\lambda \Phi_2(x,y))^\delta =  \lambda^\delta \left(\frac{y}{\lambda}+ \Phi_2(x,y)\right)^\delta, 
  \end{eqnarray*}
  we may write:
\begin{eqnarray*}
\mathcal{F}^1_{\lambda} (x,\overline{y}) &=& x+\xi+{\lambda}  \Phi_1(x,\overline{y}) - K_\omega  \ln \lambda -K_\omega \ln  \left[\left(\overline{y}+\Phi_2(x,\overline{y})\right)\right] \pmod{2\pi}\\ \\
\mathcal{F}^2_{\lambda} (x,\overline{y}) &=&  \lambda^{\delta-1} \left(\overline{y}+ \Phi_2(x,\overline{y})\right)^\delta. \\
  \end{eqnarray*}

\subsection{Reduction to a singular limit}
\label{ss: reduction}
In this subsection, we compute the singular limit of $\mathcal{F}_\lambda$ written in the coordinates $(x,\overline{y})$ studied in Subsection \ref{change_of_coordinates}, for $\lambda \in\,\, ]0, \lambda_1[$.
Let  $k: \RR^+ \rightarrow \RR$ be the invertible map defined by $$k(x)= -K_\omega \ln (x),$$
whose graph is depicted in Figure \ref{scheme3A}.
 Define now the decreasing sequence $(\lambda_n)_n$ such that, for all $n\in \NN$, we have:\\
\begin{enumerate}
\item  $\lambda_n\in\, ]0, \lambda_1[$ and \\
\item $k(\lambda_n) \equiv 0 \pmod{2\pi}$.\\
\end{enumerate}
\medbreak
 \begin{figure}[h]
\begin{center}
\includegraphics[height=10cm]{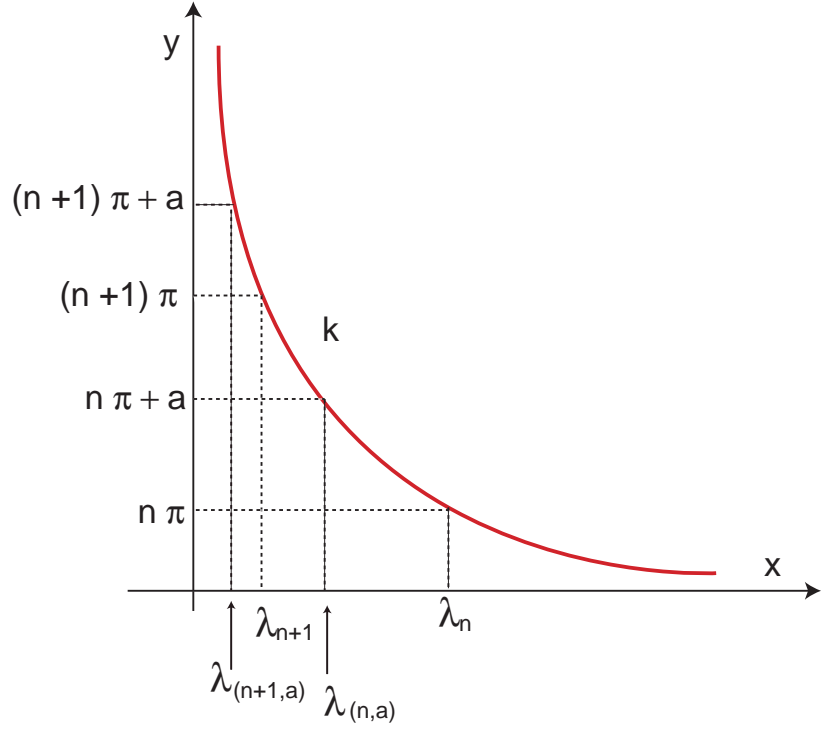}
\end{center}
\caption{\small  Graph of $k(x)= -K_\omega \ln (x)$ and illustration of the sequences $(\lambda_n)_n$ and $(\lambda_{(n, a)})_n$ for a fixed $a \in [0,2\pi[$.}
\label{scheme3A}
\end{figure}

Since $k$ is an invertible map,  for $a \in \EU^1 \equiv [0, 2\pi[$ fixed and $n\geq n_0\in \NN$, let  
\begin{equation}
\label{sequence1}
\lambda_{(a, n)}= k^{-1}(k(\lambda_n)+a)\,\,  \in \,\, ]0, \lambda_1[,
\end{equation}
as shown in Figure  \ref{scheme3A}. It is immediate to check that:
\begin{equation}
\label{sequence2}
k\left(\lambda_{(a, n)}\right)= -K_\omega \ln (\lambda_{n})+a=a \pmod{2\pi}.
\end{equation}

The following proposition establishes  the convergence of the map $\mathcal{F}_{\lambda_{(n,a)}}$  to a singular limit as $n \rightarrow +\infty$, ($\|\star\|_{\textbf{C}^r}$ represents the norm   in the $C^r$--topology for $r\geq2$):

\begin{lemma} 
\label{important lemma}
The following equality holds:
$$
\lim_{n\in \NN} \|\mathcal{F}_{\lambda_{(n,a)}} (x,\overline{y}) -( h_a(x,\overline{y}), \textbf{0})\|_{\textbf{C}^3} =0$$
where $\textbf{0}$ is the null map and
\begin{equation}
\label{circle map}
h_a(x, \overline{y})= x+a-K_\omega \ln (\overline{y} +\Phi_2(x,\overline{y})) +\xi.
\end{equation}
\end{lemma}

\begin{proof}

Using \eqref{sequence2}, note that
\begin{eqnarray*}
\mathcal{F}^1_{\lambda_{(n,a)}} (x,\overline{y}) &=&  x+\xi+{\lambda_{(n,a)}}  \Phi_1(x,\overline{y})- K_\omega  \ln \lambda_{(n,a)} -K_\omega \ln  \left[\left(\overline{y}+\Phi_2(x,\overline{y})\right)\right] \pmod{2\pi} \\
&=&  x+\xi+{\lambda_{(n,a)}}  \Phi_1(x,\overline{y})+a  -K_\omega \ln  \left[\left(\overline{y}+\Phi_2(x,\overline{y})\right)\right] \pmod{2\pi}
  \end{eqnarray*}
  and
  \begin{eqnarray*}
\mathcal{F}^2_{\lambda_{(n,a)}} (x,\overline{y}) &=&  {\lambda_{(n,a)}}^{\delta-1} \left(\overline{y}+ \Phi_2(x,\overline{y})\right)^\delta.
  \end{eqnarray*}
    Therefore, since $\lim_{n\in \NN} {\lambda_{(n,a)}}=0$ we may write:
  $$
\lim_{n\in \NN}   \mathcal{F}^1_{\lambda_{(n,a)}} (x,\overline{y}) = x+\xi+a  -K_\omega \ln  \left[\left(\Phi_2(x,0)\right)\right] \pmod{2\pi}
  $$
and 
  $$
\lim_{n\in \NN}   \mathcal{F}^2_{\lambda_{(n,a)}} (x,\overline{y}) = 0,
  $$
 and we get the result.

\end{proof}

\begin{remark}
\label{rem7.2}
The map $h_a(x) =x+\xi+a  -K_\omega \ln  \Phi_2(x,0)\equiv   \mathcal{F}^1_{\lambda_{(n,a)}} (x,0)$ has finitely  many nondegenerate critical points (by Hypothesis \textbf{(P8)}).
\end{remark}

\subsection{Verification of the hypotheses of the theory of rank-one maps.}
\label{check1}

We show that the family of flow-induced maps $\mathcal{F}_{\lambda_{(n,a)}} \equiv \mathcal{F}_{(a, \lambda_{(n,a)})}$, $a\in \EU^1$ and $n\geq n_0$, satisfies Hypotheses \textbf{(H1)--(H7)} stated in Subsection \ref{rank_one}. Since our starting point  is an attracting   network for $\lambda=0$ (cf.  Lemma \ref{attractor_lemma}), the \emph{absorbing sets} defined in Subsection 2.4 of \cite{WY} are automatic here.

\medbreak
\begin{description}
\item[\text{(H1)}]  
The first two items are immediate.  We establish the distortion bound \textbf{(H1)(3)} by studying local maps and global maps separately. Direct computation using (\ref{eqeta}) and (\ref{transition21}), implies that for every $\lambda \in (0, \lambda_1)$ and $(x, \overline{y}) \in \Out(O_2)$, one gets:

$$D\Psi_{2 \rightarrow  1}^\lambda(x,\overline{y})=\left(\begin{array}{cc} \dpt 1+ \lambda \frac{\partial \Phi_1(x,\overline{y})}{\partial x} &\dpt  \lambda \frac{\partial \Phi_1(x,\overline{y})}{\partial y} \\ \\ \dpt \lambda \frac{\partial \Phi_2(x,\overline{y})}{\partial x}& \dpt1 + \lambda \frac{\partial \Phi_2(x,\overline{y})}{\partial y}\end{array}\right)
$$
and
$$
D\eta(x,y)=\left(\begin{array}{cc} 1  &\dpt\frac{-K_\omega}{{y}}\\ \\ 0&\delta {y}^{\delta-1}\end{array}\right)
$$
Since $\mathcal{F}_{\lambda} = \eta \circ  \Psi_{2 \rightarrow  1}^{\lambda}$, taking $(X,Y)= \Psi_{2 \rightarrow  1}^{\lambda}(x,\overline{y})$, we may write:

\begin{eqnarray*}
\label{bound}
&&|\det D \mathcal{F}_\lambda(x,\overline{y})| = \\ \\
&=& | \det D\eta(X,Y)| \times |\det D \Psi_{2 \rightarrow 1}^\lambda(x,\overline{y})| \\ \\ 
&=&\left|\delta Y^{\delta-1}\right| \times \left|\left[\left(1+\lambda \frac{\partial \Phi_1(x,\overline{y})}{\partial x}\right)  \left(1    + \lambda \frac{\partial \Phi_2(x,\overline{y})}{\partial y}\right) - \lambda^2 \frac{\partial \Phi_2(x,\overline{y})}{\partial x}\frac{\partial \Phi_1(x,\overline{y})}{\partial x}\right] \right|.
\end{eqnarray*}
Since  $\delta Y^{\delta-1}>0$ (by \textbf{(P8)}),  we conclude that there exists $\lambda_2>0$ small enough such that:
$$ 
\forall \lambda \in \, ]0, \lambda_2 [, \qquad |\det D \mathcal{F}_\lambda(x,\overline{y})|  \in  \,\,  ]k_1^{-1}, k_1[ ,
$$
\bigbreak
for some $k_1>1$. This implies that hypothesis \textbf{(H1)(3)}  is satisfied.
\bigbreak
\item[\text{(H2) and (H3)}] These items follow from Lemma \ref{important lemma} where $b=\lambda_{(n,a)}$.

\bigbreak

\end{description}
\bigbreak
\bigbreak
The next proposition generalizes the results of Subsections 5.2 and 5.3 of \cite{WY} and will be used in the sequel.

\begin{proposition}[\cite{WY, WY2003}, adapted]
\label{Prop2.1WY}
 Let $A: \EU^1 \rightarrow \RR$ be a $C^3$--function with nondegenerate critical points. Then there exist $L_1$ and $\delta$ depending on $A$ such that if $L \geq  L_1$  and $B: \EU^1 \rightarrow \RR$ is a $C^3$
function with $\|\Psi\|_{\textbf{C}^2} \leq \delta$ and $\|B\|_{\textbf{C}^3}\leq 1$,
 then the family of maps
 $$ \theta \mapsto \theta+a+L(A(\theta) +B(\theta)), \qquad a\in [0,2\pi[, \qquad \theta\in \EU^1$$
satisfies \textbf{(H4)} and \textbf{(H5)}. If $L$ is sufficiently large, then  Hypothesis \textbf{(H7)} is also verified. 
\end{proposition}


\begin{description}
\item[\text{(H4) and (H5)}] 
These hypotheses are connected with the family of circle maps $$h_a: \EU^1 \rightarrow \EU^1$$ defined in Remark \eqref{rem7.2}. 
Taking into account the Proposition \ref{Prop2.1WY}, the family $$
h_a(x)=x+\xi +a+K_\omega \log(\Phi_2(x,0)) \qquad a\in [0,2\pi[
$$
satisfies Properties \textbf{(H4)} and \textbf{(H5)}.  

\bigbreak
\item[\text{(H6)}] The computation follows  from direct computation using  the expression of $\mathcal{F}^1_\lambda(x, 0)$. Indeed, for each $x\in C_{a^\star}$ (set of critical points of the Misiurewicz-type map $h_{a^\star}$ whose existence is ensured in \textbf{(H4)}), we have 
$$
\frac{d}{dy} \mathcal{F}^1_{(a^\star,0)}(x,y) |_{y=0} =1+ K_\omega \,\,\frac{1+ \frac{\dpt \partial\Phi_2(x,y)}{\partial y}|_{y=0}}{1+\Phi_2(x,y)}\neq 0,
$$
as a consequence of  \textbf{(P8)}.
\bigbreak

\item[\text{(H7)}] It follows from Proposition \ref{Prop2.1WY} if $K_\omega $ is large enough ($\Rightarrow$ the ``big lobe''  of \cite{SST, TS1986} is large enough).

\end{description}

\bigbreak

Since the family $\mathcal{F}_{(a,b)}$ for $b=\lambda_{(n,a)}$ satisfies \textbf{(H1)--(H7)} of  \cite{WY} (revisited in  Theorem \ref{th_review}) then, for $\lambda_0=\min\{\lambda_1, \lambda_2\}>0$, for $K_\omega >K_\omega^0$, there exists a subset $\tilde\Delta \in  [0, \lambda_0]$ with positive Lebesgue
measure such that for $\lambda\in \tilde\Delta$, the map $\mathcal{F}_{\lambda}$ admits a strange attractor 
\begin{equation}
\label{Omega}
 \Omega \subset \bigcap_{m=0}^{+\infty}  \mathcal{F}_{\lambda}^m(\mathcal{D})
 \end{equation}
   that supports a unique ergodic SRB measure $\nu$. The orbit of Lebesgue almost all points in ${\Omega}$ has positive Lyapunov exponent  and is asymptotically distributed according to $\nu$.

\begin{remark}

From  the reasoning of Section 3 of \cite{WO}, we also have:  
\begin{equation}
\label{abunda2}
\liminf_{r\rightarrow 0^+}\, \,  \frac{ \emph{Leb} \left\{\lambda \in [0,r]: \mathcal{F}_\lambda  \text{  has a strange attractor with a SRB measure}\right\}}{r}  >0.
\end{equation}
where \emph{Leb} denotes the one-dimensional Lebesgue measure. This means that the existence of strange attractors for $ \mathcal{F}_\lambda$ is an \emph{abundant phenomenon} in the terminology of \cite{MV93}. 
\end{remark}

\begin{remark}
Theorem B remains valid if  we take into account the high order terms of \eqref{local_v} and \eqref{local_w}. The proof's arguments run along the same lines to those of \cite{WO}. 
\end{remark}

\subsection{Ergodic remark}

For the sake of completeness, we decided to add this statistical remark on the article based on the results of \cite{WY2013}.  
The dynamical system $(\mathcal{F}_{\lambda}, \nu) $ has \emph{exponential decay of correlations} for H\"older continuous observables: given an H\"older exponent $\eta$, there exists $\tau_\eta<1$ such that for all Holder maps $\varphi$, $\psi: \Omega \rightarrow \RR$ with H\"older exponent $\eta$, there exists $K(\varphi, \psi)$ such that for all $m\in\NN$, we have:
$$
\left|\int (\varphi \circ \mathcal{F}^{m}_\lambda)\psi d\nu - \int \varphi d\nu \int \psi d \nu\right| \leq K(\varphi,\psi) \tau^m.
$$
A further analysis is outside the scope of the present work. 

\section{Proof of Theorem \ref{thm:C} }
\label{Proof: ThC}

In this section, we prove the existence of a superstable sink for a sequence of $\lambda$-values for the one-parameter family  \eqref{general2.1}. For $q\in \NN$ and $a\in [0,2\pi]$, let $C(h_a)=\{c^{(1)},...,c^{(q)}\}$ be the set of critical points of $h_a$ defined in Remark \ref{rem7.2}. As the result of  the  rank-one theory, for $\lambda$ small, all points lying in $\Omega$ (see \eqref{Omega})  have at least one contracting direction.

 \medbreak
Let $a^\star \in [0, 2\pi]$ be such that the limit family $h_{a^\star}$ is a  Misiurewicz-type map (see Subsection \ref{check1}).    
Therefore, from Subsection \ref{Misiurewicz-type map}, we may define the numbers:\\
\begin{itemize}
\item[$\delta_0$:]   size of the critical set of $h_{a^\star}$; \\
\item[$\lambda_0$:]   logarithmic expansion of the orbits outside the critical set; \\
\item[$b_0$:]    {pre-factor associated to the} exponential growth outside the critical set. \\

\end{itemize}
\medbreak
For fixed $\lambda<\lambda_0/5$ and   $\alpha>0$ small, let  $\Delta(\lambda, \alpha)$ be the set of $a\in [0,2\pi]$ for which the following conditions hold for $c\in C(h_a)$ and $n\in \NN$:\\
\begin{description}
\item[\textbf{(CE1)}]   $dist(h_a^n(c), C(h_a)) \geq \min\{\delta_0/2,  2\exp (-\alpha n)\}$; \\
\item[\textbf{(CE2)}]  $|(h_a^n)'(h_a(c))|\geq 2b_0\delta_0 \exp (\lambda n)$.\\
\end{description}
These assertions are usually called by $(\lambda, \alpha)$ \emph{Collet-Eickmann conditions}. The next technical result says that $a^\star$ is a Lebesgue \emph{density point} of $\Delta(\lambda, \alpha)$. This means that,   under precise conditions on $\alpha$ and $\lambda$, we may  easily find elements $\hat{a}\in [0,2\pi]$  for which the  $(\lambda, \alpha)$ Collet-Eickmann conditions are verified.

\begin{lemma} [\cite{WY2006}, adapted]
\label{abundance1}
For $a^\star \in [0, 2\pi]$ as above, the following equality holds for  fixed $\lambda<\lambda_0/5$ and   $\alpha>0$ small:
$$
\liminf_{r\rightarrow 0^+} \frac{\text{Leb} \{\Delta(\lambda, \alpha) \cap [a^\star, a^\star+r]\}}{r}=1.
$$
\end{lemma}

  \begin{figure}[h]
\begin{center}
\includegraphics[height=6.2cm]{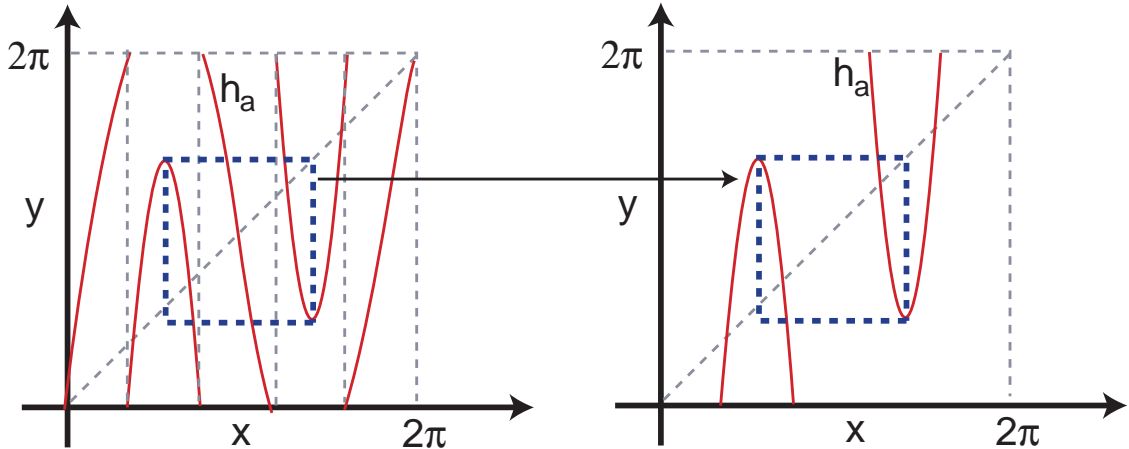}
\end{center}
\caption{\small  Graph of the map $h_a$ for $K_\omega \gg K_\omega^0$ with $q=2$ (number of critical points). Indicated is a superstable periodic orbit of period 2  as it contains a critical point. }
\label{periodic_orbit}
\end{figure}

  To  prove Theorem  \ref{thm:C}, we make use of the following result:
  \bigbreak
  \bigbreak
\begin{proposition}[\cite{OW10}, adapted]
\label{OW_review}
Let $\{h_a: a\in [0, 2\pi]\}$ be a family of maps  as above such that: \\
\begin{enumerate}
\item[(i)] there exists $a^\star$ such that  $h_{a^\star}$ is a Misiurewicz-type map; \\
\item[(ii)] $[0, 2\pi]\subset h_{a^\star}([c^{(i)}, c^{(i+1)}])$, for all  $i\in \{1, ...q\}$;\\
\item[(iii)] $\exp(\lambda_0)>\ln 10$.\\
\end{enumerate}
Hence, for every $\alpha>0$ sufficiently small and every $\hat{a}\in \Delta(\lambda, \alpha)$ sufficiently close to $a^\star$, there exists $(a_n)_{n\in \NN}$ converging to $\hat{a}$ for which $h_{a_n}$ admits a superstable periodic orbit. 
\end{proposition}

Taking into account the proof of  Theorem \ref{thm:B}, for $K_\omega>0$ large enough, we know that conditions (i) -- (iii) of Proposition \ref{OW_review} hold. 
Therefore, we conclude that for every $\alpha>0$  small and for every $\hat{a}\in \Delta(\lambda, \alpha)$ close to $a^\star$, there exists a sequence $(a_n)_{n\in \NN}$ converging to $\hat{a}$ for which $h_{a_n}$ admits a superstable sink. 
By \eqref{sequence2}, for $n\in \NN$,  we have  $$\lambda_n= \exp \left(\frac{a_n-2n\pi}{K_\omega}\right).$$ 
It is easy to see that $\dpt \lim_{n\rightarrow +\infty}\lambda_n=0$. For this sequence of values, the flow of $f_{\lambda_n}$   has a superstable 2-periodic orbit. The way this periodic orbit, which is a critical point for $h_a$, is obtained is illustrated in Figure \ref{periodic_orbit}.

\section{Mechanism to create rank-one attractors: \\a geometrical interpretation}
\label{s:mechanism}
This manuscript serves to make a bridge between the theory of rank-one attractors developed in \cite{WY2001, WY, WY2003} and the theory of heteroclinic bifurcations involving saddle-foci or periodic solutions \cite{LR, LR2016, MO15, Ott2008, WO, Wang_2016}. In addition, we exhibit a mechanism to construct \emph{rotational horseshoes} (see \S \ref{ss: rotational horseshoe}) and \emph{strange attractors}.  This work may be seen as  the natural extension of \cite{Rodrigues2019}, where a two-parameter family of vector fields has been considered to model  bifurcations associated to certain unfoldings of Hopf-zero singularities \cite{BIS, Gaspard}. 
\medbreak
 \begin{figure}[h]
\begin{center}
\includegraphics[height=7cm]{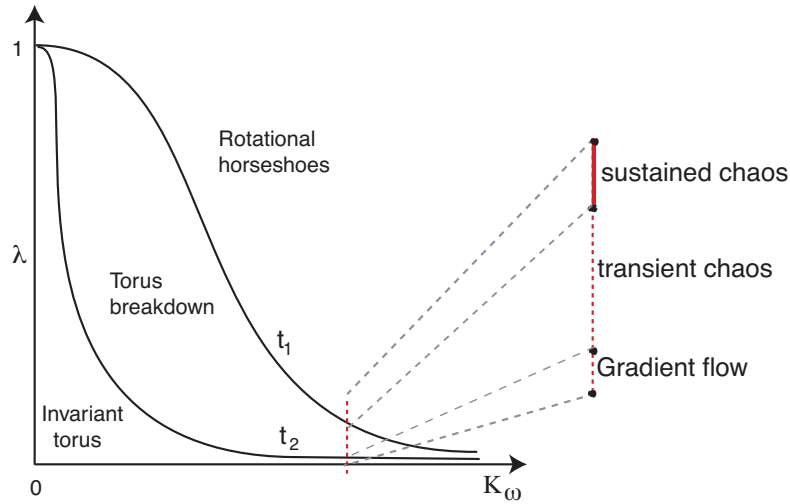}
\end{center}
\caption{\small   Adapted bifurcation diagram of \cite{Rodrigues2019} for the family \eqref{general2.1}: below the graph of $t_2$, the dynamics is governed by an attracting invariant 2-dimensional torus. Above the graph of $t_1$, one observes rotational horseshoes.  In between, rotational horseshoes exist but they may be not observable in numerics;  the dynamics is probably dominated either by a sink or bistability between a sink and an attracting torus. }
\label{scheme3}
\end{figure}
In this section, we compare and discuss our results with previous works in the literature.
Borrowing the arguments used in \cite{Rodrigues2019}, we may draw, in the first quadrant,  two smooth curves, the graphs of $t_1$ and $t_2$ drawn in Figure \ref{scheme3}, such that:
\begin{enumerate}
\medbreak
\item for all $\omega \in \RR^+_0$, we have $t_2(\omega)\leq t_1(w)$;
\medbreak
\item the region below the graph of $t_2$  corresponds to flows having an invariant and attracting torus with zero topological entropy  (attracting invariant torus);
\medbreak
\item the region between the graphs of $t_1$ and $t_2$ corresponds to flows having rotational horseshoes, which might not be observable in numerical simulations;
\medbreak
\item for $K_\omega$ sufficiently large, the region above the graph of $t_1$ correspond to flows exhibiting  strange attractors. 
\end{enumerate}

\subsection{From an invariant curve to a strange attractor}
Forgetting temporarily its connection to  equation \eqref{circle map}, we might  think of $$h_a(x)=x+\xi+a+K_\omega \ln (\Phi_2(x,0)), $$
$x\in \EU^1$ and $ a  \sim - K_\omega \ln \mu \pmod{2\pi}$,  as an abstract circle
map. 
The emergence of this map in the first return map to a Bykov network $\Gamma$ is expected.  The rotation forced by the existence of saddle-foci causes a trajectory to be thrown out in all directions in the associated eigenspace on successive approaches to the equilibria \cite{AC}.  Several types of behaviours are known to
be typical for $h_a$:
\medbreak
\begin{description}
\item[(a)] If $h_a$ is a diffeomorphism, the classical theory of Poincaré and Denjoy
may be applied. We point out a resemblance between
$h_a$ and the  family of circle maps studied in \cite{Aronson}. Because of strong normal contraction, invariant curves are shown to exist independent of
rotation number. 
\medbreak

\item[(b)] If $h_a$ is non-invertible, in general the rotation number of a given initial condition is not unique and hence an interval of rotation numbers may exist, which implies chaos \cite{MT}.  According to Section 3 of \cite{WY}, two types of dynamical behaviours are known to be \emph{prevalent}:  
\begin{description}
\item[(b1)] maps with sinks or
\item[(b2)] maps with absolutely continuous invariant measures. 
\end{description}
\end{description}
\medbreak 


 The passage of Cases \textbf{(a)} and \textbf{(b1)}  in one dimension to uniform hyperbolicity in two dimensions is
relatively simple. 
The passage of the   case \textbf{(b2)} to two dimensions is the core of \cite{WY2001} and has been  revisited in Section \ref{rank_one}  of the present paper. 

\subsection{Phenomenological description}
\label{ss:at}
 For $\lambda, K_\omega$ small enough, the flow of (\ref{general2.1}) has an attracting  hyperbolic torus, which persists under small smooth perturbations.  This is consistent with the results stated in \cite{AS91, AHL2001} about the existence of an attracting curve for the map $\mathcal{F}_\lambda$. The initial deformation of $\eta$ (see \eqref{eqeta}) on the curve $\dpt \In(O_1) \, \cap \, \Phi_2(x,0)$ is suppressed by the contracting force, and the attractor is a non-contractible closed curve obtained by applying the  \emph{Afraimovich's Annulus Principle} \cite{AHL2001}.   
 
 For a fixed $\lambda>0$, if the twisting number $K_\omega$ increases, points at $\In(O_1) \cap \Phi_2(x,0)$ at different distances from $W^s(O_1)$  rotate at different speeds. The   attracting curve starts to disintegrate into a finite collection of periodic saddles and sinks; this phenomenon 
is part of the \emph{torus-breakdown theory} \cite{AS91} and occurs within an \emph{Arnold tongue}  \cite{Aronson}.  Once  the rotational horseshoes   develop, they persist and  correspond to what the authors of \cite{WY} call \emph{transient chaos}.   The \emph{rotational horseshoes} found in \cite{Rodrigues2019}  and the associated invariant manifolds are (in general) not observable in numerical  simulations because they form a set of Lesbegue measure zero; almost all solutions will be trapped by a \emph{sink}.

  \begin{figure}[h]
\begin{center}
\includegraphics[height=8cm]{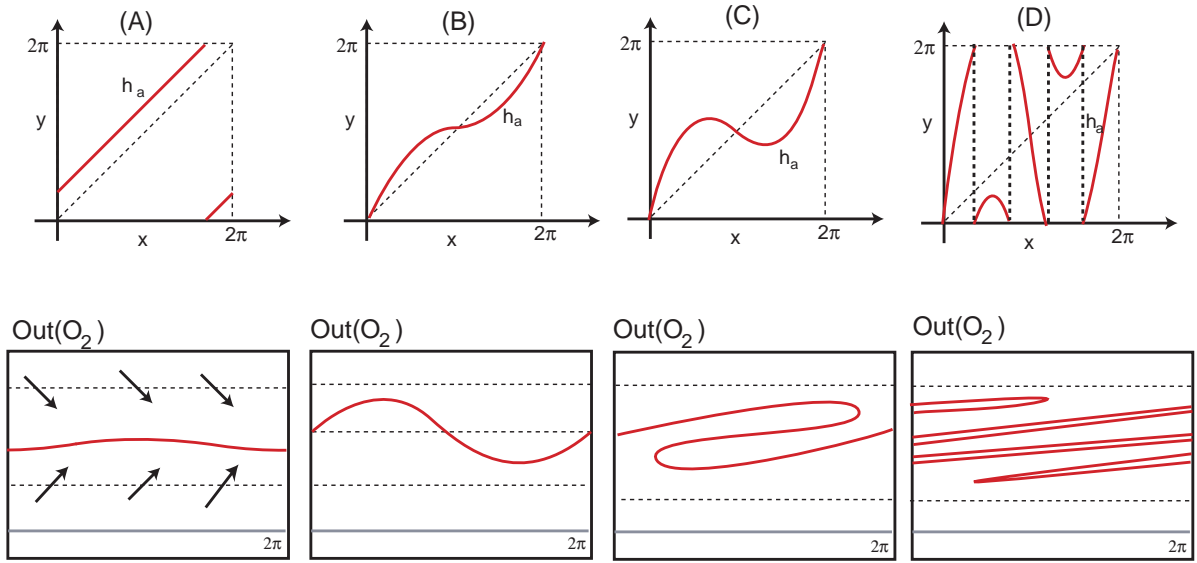} 
\end{center}
\caption{\small  Illustration (based on plots obtained by using \emph{Maple}) of $h_a(x) \pmod{2\pi}$ for $a=0.1$, $\xi=0$ and $\Phi_2(x,0)=\sin x$ ($\Rightarrow q=2$). (a) $K_\omega \gtrsim0$ -- attracting torus. (b) $K_\omega=0.5$ -- attracting torus inside a resonance wedge with a saddle and a sink. (c):  $K_\omega=2$ -- torus breakdown.  (d): $K_\omega=5$ -- mixing properties (``big lobe'' of  \cite{SST}).  }
\label{scheme2}
\end{figure}

As the twisting number $K_\omega$ gets larger, the  deformation on the curve $\In(O_1) \cap \Phi_2(x,0)$ is exaggerated further, giving rise to rank-one attractors created by stretch-and-fold type actions  -- \emph{sustained chaos}  \cite{WY}. Iterates do not escape and wander around the full torus (``large'' strange attractors).  These two phenomena (stretch-and-fold) are due to the network \emph{attracting} features combined with the presence of \emph{complex eigenvalues}, which force a  strong distortion.

\subsection{Interpretation of Figure \ref{scheme2}}

 The evolution of $h_a(x)=x+\xi+ a+K_\omega \ln (\Phi_2(x,0))$, when $\xi, \lambda>0$ are fixed and $K_\omega$ varies, is suggested in Figure \ref{scheme2}. In (A) and (B), we deduce the existence of an invariant curve (a torus in the flow of \eqref{general2.1}). The existence of periodic orbits within the torus depends on the rotation number. 
In case (B) we observe the existence of the two fixed points, suggesting that the chosen parameters are within a resonant tongue. In (C), the map $h_a$ is not a diffeomorphism meaning that  the invariant curve broke up; it corresponds to the point when the unstable manifold of the saddle (in the ``Arnold tongue'' \cite{Aronson, Rodrigues2021b}) turns around.  
In case (D), Property \textbf{(H7)} holds, meaning that the unstable manifold of a saddle within the ``Arnold tongue'' crosses each leaf of the stable foliation of the saddles of the torus's ghost.  This corresponds to what Turaev and Shilnikov \cite{TS1986} call a \emph{big lobe}. This non-negative quantity determines the magnitude between the given minimal value $h_a$ and the preceding maximal one. For $a>0$ fixed, as the twisting number increases, the critical points of $h_a(c^{(1)})$ and $h_a(c^{(2)})$  move in opposite directions at rate $\frac{1}{K_\omega}$ \cite{WY, WY2003, WY2008}.
\medbreak

\subsection{Concluding remarks}
\label{s:Discussion2}
This article studies global bifurcations associated to the one-para\-meter family $\dot{x}=f_{\lambda}(x)$, defined on a 3-dimensional sphere and satisfying Properties \textbf{(P1)--(P8)}. For $\lambda=0$, there is an attracting network containing a 2-dimensional connection (continua of solutions \cite{AC}) associated to two saddle-foci with different Morse indices, whose  invariant manifolds do not intersect for $\lambda>0$. 
 These  2-dimensional manifolds are forced to fold uniformly   within the network's basin of attraction.

 The novelty of this article is that, assuming that the \emph{twisting number} $K_\omega$ is large enough, rank-one attractors arise and wind around the full 2-dimensional torus whose existence is stated in Theorem B of \cite{Rodrigues2019}.   
The 2-dimensional invariant  manifold $W^u(O_2)$  plays the essential role in the construction of the global attractor where the SRB measure    is supported.  The strange attractors in the present paper are nonuniformly hyperbolic and structurally unstable (they are not \emph{robustly transitive} as the geometric Lorenz model).

This work is the natural continuation of \cite{Rodrigues2019}. 
Using the powerful theory of Wang and Young, we  formulate checkable hypotheses (in terms of the vector field \eqref{general2.1}) under which the first return map to global cross section  admits a strange attractor and superstable periodic orbits. The chaos is
observable (it has a strange attractor), ``large'' (is not confined to a small portion of the phase space)  and abundant (condition \eqref{abunda2} holds in the space of parameters).  As far as we know, it is the first time that this theory is applied to an autonomous family. 

Results in this manuscript go further in the analysis of unfoldings of the heteroclinic attractors. These families may behave  periodically, quasi-periodically or chaotically, depending on specific character of the perturbation. Techniques used in the present article are valid for more general heteroclinic networks containing 2-dimensional connections that are pulled apart an also  in the context of periodically perturbed networks. The study of the ergodic consequences of this dynamical scenario is the natural continuation of the present work.

\section*{Acknowledgements}
The  author is grateful to Isabel Labouriau for the fruitful discussions. The author is indebted to the reviewer for the corrections and suggestions which helped to improve the readability of this manuscript.

\appendix


\begin{thebibliography}{777}

\bibitem{AS91} V.S. Afraimovich, L.P. Shilnikov, \emph{On invariant two-dimensional tori, their breakdown and stochasticity} in: Methods of the Qualitative Theory of Differential Equations, Gor'kov. Gos. University (1983), 3--26. Translated in: Amer. Math. Soc. Transl., (2), vol. 149 (1991) 201--212.

\bibitem{AHL2001} V.S. Afraimovich, S-B Hsu, H. E. Lin, \emph{Chaotic behavior of three competing species of May– Leonard model under small periodic perturbations}. Int. J. Bif. Chaos, 11(2) (2001) 435--447.

\bibitem{Aguiar_tese} M. Aguiar, \emph{Vector fields with heteroclinic networks}, Ph.D. thesis, Departamento de Matem\'atica Aplicada, Faculdade de Ci\^encias da Universidade do Porto, (2003).

\bibitem{Algaba} A. Algaba,  F. Fern\'andez-S\'anchez, M. Merino, A. Rodriguez-Luis, \emph{Analysis of the $T$-point - Hopf bifurcation with $\mathbb{Z}_2$-symmetry: application to Chua's equation}, International Journal of Bifurcation and Chaos, 20(04) (2010) 979--993.


\bibitem{Aronson} D. Aronson, M. Chory, G. Hall, R. McGehee, \emph{ Bifurcations from an invariant circle for two-parameter families of maps of the plane: a computer-assisted study}, Communications in Mathematical Physics, 83(3)  (1982) 303--354.

\bibitem{AC} P. Ashwin,  P. Chossat, \emph{Attractors for robust heteroclinic cycles with continua of connections}, Journal of Nonlinear Science, 8(2) (1998) 103--129.

 

\bibitem{BIS} I. Baldom\'a, S. Ib\'a\~nez, T. Seara, \emph{Hopf-Zero singularities truly unfold chaos}, Commun. Nonlinear Sci. Numer. Simul. 84  (2020), 105162.



\bibitem{BC91} M. Benedicks, L. Carleson, \emph{The dynamics of the H\'enon map},  Annals of Mathematics  133(1) (1991) 73--169.

\bibitem{BY93} M. Benedicks, L. Young, \emph{Sinai-Bowen-Ruelle measures for certain H\'enon maps}, Invent. Math. 112 (1993), 541--576.


 \bibitem{BST98} H. Broer, C. Simó, J. C. Tatjer, \emph{Towards global models near homoclinic tangencies of dissipative diffeomorphisms} Nonlinearity 11 (1998) 667--770.


\bibitem{Bykov00} V.V.~Bykov, \emph{Orbit Structure in a neighborhood of a separatrix cycle containing two saddle-foci.} Amer. Math. Soc. Transl. 200 (2000) 87--97.

\bibitem{CastroR2020} M. L. Castro, A. A. P. Rodrigues, \emph{Torus-breakdown near a Bykov attractor with rotational symmetry},  International Journal of Bifurcation and Chaos, 31(10), (2021) 2130029. 

\bibitem{Deng1} B.~Deng, \emph{The Shilnikov Problem, Exponential Expansion, Strong $\lambda$--Lemma, $C^1$ Linearisation and Homoclinic Bifurcation}, J. Diff. Eqs 79 (1989) 189--231.

 

\bibitem{DIKS} F.~Dumortier, S.~Ib\'a\~nez, H.~Kokubu, C.~Sim\'o, \emph{About the unfolding of a Hopf-zero singularity.} Discrete Contin. Dyn. Syst. 33(10) (2013) 4435--4471.

\bibitem{Gaspard} P. Gaspard, \emph{ Local birth of homoclinic chaos},  Physica D: Nonlinear Phenomena, 62(1-4) (1993) 94--122.


 


\bibitem{GH} J.~Guckenheimer, P.~Holmes, \emph{Nonlinear Oscillations, Dynamical Systems, and Bifurcations of Vector Fields.} Applied Mathematical Sciences 42, Springer-Verlag, (1983).

\bibitem{He76} M. H\'enon, \emph{A two dimensional mapping with a strange attractor}, Comm. Math. Phys. 50 (1976), 69--77.

 

\bibitem{Ja81} M. Jakobson, \emph{Absolutely continuous invariant measures for one parameter families of one-dimensional maps}, Comm. Math. Phys. 81 (1981), 39--88.

\bibitem{KLW} J. Knobloch, J.S.W.~Lamb, K.N.~Webster, \emph{Using Lin's method to solve Bykov's problems.} J. Diff. Eqs. 257(8) (2014) 2984--3047.

\bibitem{Homb2002} A.J. Homburg, \emph{Periodic attractors, strange attractors and hyperbolic dynamics near homoclinic orbits to saddle-focus equilibria.} Nonlinearity 15 (2002) 1029--1050.

\bibitem{HS} A.J. Homburg,  B. Sandstede,  \emph{Homoclinic and heteroclinic bifurcations in vector fields}, Handbook of dynamical systems 3 (2010): 379--524.

\bibitem{LR} I.S.~Labouriau, A.A.P.~Rodrigues, \emph{Global generic dynamics close to symmetry.} J. Diff. Eqs. 253(8) (2012) 2527--2557.

\bibitem{LR2015} I.S.~Labouriau, A.A.P.~Rodrigues, \emph{Dense heteroclinic tangencies near a Bykov cycle}, J. Diff. Eqs. 259(12) (2015) 5875--5902.

\bibitem{LR2016} I.S.~Labouriau, A.A.P.~Rodrigues, \emph{Global bifurcations close to symmetry.} J. Math. Anal. Appl. 444(1) (2016) 648--671.

\bibitem{LWY} K. Lu, Q. Wang,  L.-S.Young, \emph{Strange attractors for periodically forced parabolic equations} (Vol. 224)  American Mathematical Soc. (2013).

\bibitem{MT} R. Mackay, C. Tresser, \emph{Transition to topological chaos for circle maps}, Physica D: Nonlinear Phenomena 19.2 (1986) 206--237.

\bibitem{M} M. Misiurewicz, \emph{Absolutely continuous measures for certain maps of an interval}, Inst. Hautes \'Etudes Sci. Publ. Math. 53 (1981) 17--51.

\bibitem{MO15} A. Mohapatra, W. Ott, \emph{Homoclinic Loops, Heteroclinic Cycles, and Rank One Dynamics},  SIAM Journal on Applied Dynamical Systems, 14(1) (2015) 107--131.

\bibitem{MV93} L. Mora, M. Viana, \emph{Abundance of strange attractors}, Acta Math. 171(1) (1993) 1--71.

\bibitem{Ott2008} W. Ott, \emph{Strange attractors in periodically-kicked degenerate Hopf bifurcations}. Communications in mathematical physics, 281(3), (2008) 775--791.

\bibitem{OS2010} W. Ott, M. Stenlund, \emph{From limit cycles to strange attractors}. Communications in Mathematical Physics, 296(1), (2010) 215--249.
 
\bibitem{OW10} W. Ott,  Q. Wang, \emph{Periodic attractors versus nonuniform expansion in singular limits of families of rank one maps}, Discrete \& Continuous Dynamical Systems-A 26.3 (2010): 1035.

\bibitem{OS} I.M.~Ovsyannikov, L.P. Shilnikov, \emph{On systems with a saddle-focus homoclinic curve.} Math. USSR Sb. 58 (1987) 557--574.

\bibitem{PPS} A. Passeggi, R. Potrie, M. Sambarino, \emph{Rotation intervals and entropy on attracting annular continua}, Geometry \& Topology 22(4) (2018) 2145--2186.

 

\bibitem{Rodrigues3} A.A.P.~Rodrigues, \emph{Repelling dynamics near a Bykov cycle.} J. Dyn. Diff. Eqs. 25(3) (2013) 605--625.

\bibitem{Rodrigues2019} A.A.P.~Rodrigues, \emph{Unfolding a Bykov attractor:  from an attracting torus to strange attractors}, J Dyn Diff Equat (2020) https://doi.org/10.1007/s10884-020-09858-z.

\bibitem{Rodrigues2021} A.A.P.~Rodrigues, \emph{Abundance of Strange Attractors Near an Attracting Periodically Perturbed Network}, SIAM J. Appl. Dyn. Syst., 20(1), 541--570, 2021.

\bibitem{Rodrigues2021b} A.A.P.~Rodrigues, \emph{Dissecting a Resonance Wedge on Heteroclinic Bifurcations}, J Stat Phys 184, 25 (2021) https://doi.org/10.1007/s10955-021-02811-4

 

\bibitem{RT71} D. Ruelle, F. Takens,  \emph{On the nature of turbulence}. Commun.Math.Phys. 20(3) (1971) 167--192.


 

\bibitem{SNN95} A. Shilnikov,  G. Nicolis, C. Nicolis, \emph{Bifurcation and predictability analysis of a low-order atmospheric circulation model}, International Journal of Bifurcation and Chaos 5.06 (1995): 1701--1711.

\bibitem{SST} A. Shilnikov, L. Shilnikov, D. Turaev, \emph{ On Some Mathematical Topics in Classical Synchronization: a Tutorial},  International Journal of Bifurcation and Chaos, 14(07) (2004) 2143--2160.



 


\bibitem{TS1986} D.Turaev, L.P.Shilnikov, \emph{Bifurcations of quasiattractors torus-chaos}, Mathematical mechanisms of turbulence (1986) 113--121.

\bibitem{Wang_2016} Q. Wang, \emph{Periodically forced homoclinic loops to a dissipative saddle}, J. Diff. Eqs. 260 (2016) 4366--4392.

\bibitem{WO} Q. Wang, W. Ott, \emph{Dissipative homoclinic loops of two-dimensional maps and strange attractors with one direction of instability}, Communications on Pure and Applied Mathematics 64.11 (2011) 1439--1496.

\bibitem{WY2001} Q. Wang,  L.-S. Young, \emph{ Strange attractors with one direction of instability}. Commun. Math. Phys. 218 (2001) 1--97.


\bibitem{WY} Q. Wang, L.-S.  Young,  \emph{From Invariant Curves to Strange Attractors}, Commun. Math. Phys. (2002) 225--275.

\bibitem{WY2003} Q. Wang, L.-S.  Young,  \emph{Strange Attractors in Periodically-Kicked Limit Cycles
and Hopf Bifurcations}, Commun. Math. Phys. 240 (2003) 509--529.

\bibitem{WY2006} Q. Wang,  L.-S.  Young,  \emph{Nonuniformly expanding 1D maps}, Communications in mathematical physics 264(1) (2006), 255--282.

\bibitem{WY2008}  Q. Wang, L.-S.  Young,  \emph{Toward a theory of rank one attractors}, Ann. of Math. (2) 167, (2008) 349--480. 

\bibitem{WY2013} Q. Wang,  L.-S.  Young,  \emph{Dynamical profile of a class of rank-one attractors}. Ergodic Theory and Dynamical Systems, 33(4) (2013) 1221--1264.

\bibitem{Yo98} L.-S. Young, \emph{Statistical properties of dynamical systems with some hyperbolicity,} Ann. Math. 147 (1998), 585--650.



\end{thebibliography}
\end{document}